\documentclass[11pt]{amsart}

\usepackage[draft]{changes}

\definechangesauthor[color=red]{pe}
\definechangesauthor[color=red]{xm}
\definechangesauthor[color=red]{hx}
\newcommand{\stkout}[1]{\ifmmode\text{\sout{\ensuremath{#1}}}\else\sout{#1}\fi}
\setdeletedmarkup{\stkout{#1}}
\colorlet{Changes@Color}{red}

\usepackage{amssymb}

\usepackage{graphics}
\usepackage{graphicx}
\usepackage{amsmath}
\usepackage{amsthm}
\usepackage{amsfonts}
\usepackage{mathtools}
\usepackage{xcolor}
\usepackage[english]{babel}
\usepackage[margin=1.0in]{geometry}
\usepackage[colorlinks=true,
        linkcolor=blue]{hyperref}
\usepackage{bbm}
\usepackage{verbatim}
\usepackage{extarrows}
\usepackage{blkarray}

\usepackage[T1]{fontenc}

\numberwithin{equation}{section}

\newtheorem{prop}{Proposition}
\newtheorem{lemma}[prop]{Lemma}

\newtheorem{thm}[prop]{Theorem}
\newtheorem*{thm*}{Theorem}
\newtheorem{cor}[prop]{Corollary}
\newtheorem{conj}[prop]{Conjecture}

\numberwithin{prop}{section}

\newtheorem{defn}[prop]{Definition}
\newtheorem*{defn*}{Definition}
\theoremstyle{definition}

\newtheorem*{ex*}{Example}
\newtheorem{rmk}[prop]{Remark}

\definecolor{c1}{rgb}{0.2,0.4,0.5}
\definecolor{c2}{rgb}{0.1,0.3,0.5}
\definecolor{c3}{rgb}{0.2,0.7,0.5}
\usepackage{tikz}

\def \k {K\"ahler }

\newcommand{\oo}[1]{\overline{#1}}

\newcommand{\dbar}{\oo\partial}

\DeclareFontFamily{U}{MnSymbolC}{}
\DeclareSymbolFont{MnSyC}{U}{MnSymbolC}{m}{n}
\DeclareFontShape{U}{MnSymbolC}{m}{n}{
	<-6>  MnSymbolC5
	<6-7>  MnSymbolC6
	<7-8>  MnSymbolC7
	<8-9>  MnSymbolC8
	<9-10> MnSymbolC9
	<10-12> MnSymbolC10
	<12->   MnSymbolC12}{}
\DeclareMathSymbol{\intprod}{\mathbin}{MnSyC}{'270}

\DeclareMathOperator{\supp}{supp}

\DeclareMathOperator{\Ric}{Ric}

\DeclareMathOperator{\Rm}{Rm}

\DeclareMathOperator{\Vol}{Vol}

\begin{document}

\title[]{Bergman kernels over polarized K\"ahler manifolds, Bergman logarithmic flatness, and a question of Lu--Tian}

\begin{abstract}
Let $M$ be a complete \k manifold, and let $(L, h) \to M$ be a positive line bundle inducing a \k metric $g$ on $M$. We study two Bergman kernels in this setting: the Bergman kernel of the disk bundle of the dual line bundle $(L^*, h^*)$, and the Bergman kernel of the line bundle $(L^k, h^k)$, $k\geq 1$, twisted by the canonical line bundle of $(M, g)$. We first prove a localization result for the former Bergman kernel. Then we establish a necessary and sufficient condition for this Bergman kernel to have no logarithmic singularity, expressed in terms of the Tian--Yau--Zelditch--Catlin type expansion of the latter Bergman kernel. This result, in particular, answers a question posed by Lu and Tian. As an application, we show that if $(M, g)$ is compact and locally homogeneous, then the circle bundle of $(L^*, h^*)$ is necessarily Bergman logarithmically flat.

\end{abstract}



\author [Ebenfelt]{Peter Ebenfelt}
\address{Department of Mathematics, University of California at San Diego, La Jolla, CA 92093, USA}
\email{{pebenfelt@ucsd.edu}}

\author[Xiao]{Ming Xiao}
\address{Department of Mathematics, University of California at San Diego, La Jolla, CA 92093, USA}
\email{{m3xiao@ucsd.edu}}

\author [Xu]{Hang Xu}
\address{School of Mathematics (Zhuhai), Sun Yat-sen University, Zhuhai, Guangdong 519082, China}
\email{{xuhang9@mail.sysu.edu.cn}}

\thanks{The first author was supported in part by the NSF grant DMS-2453134 and DMS-2154368. The second author was supported in part by the NSF grant DMS-2045104. The third author was supported in part by the NSFC grant No. 12201040 and NSF Guangdong Province grant No. 2025A1515012071. }

\maketitle

\section{Introduction}

Let $\Omega$ be a $C^{\infty}$-smoothly bounded, strongly pseudoconvex domain in $\mathbb{C}^n$. We denote by $K_{\Omega}=K_{\Omega}(z, \oo{z})$ its Bergman kernel and  let $\rho$ be a defining function for $\Omega$ such that $\Omega=\{\rho>0\}$. By a classical result of Fefferman \cite{Fe}, $K_{\Omega}$ has the following asymptotic expansion in $\Omega$,
\begin{equation}\label{Fefferman}
	K_{\Omega}=\frac{\phi}{\rho^{n+1}}+\psi\log\rho,
\end{equation}
where $\phi$ and $\psi$ are $C^{\infty}$-smooth functions on $\oo{\Omega}$.
The strong singularity $\phi\rho^{-n-1}$ and the logarithmic singularity $\psi\log\rho$ in this expansion
are determined by the local CR geometry of the boundary $\partial \Omega$ \cite{Fe, BEG94, Hi00}. Moreover, it is known that the function $\phi$ does not vanish on $\partial\Omega$ whereas $\psi$ may. In fact, it may happen that $\psi$ vanishes to infinite order along $\partial\Omega$ (i.e., $\psi=0 \mod O(\rho^{\infty})$). When this occurs, we say that $\partial\Omega$ is \emph{Bergman logarithmically flat}. Since Bergman logarithmic flatness depends only on the local CR geometry of $\partial \Omega$, we may introduce this notion for any (always $C^{\infty}$-smooth throughout this paper) strongly pseudoconvex real hypersurface.
More precisely, if $\Sigma$ is a germ of a strongly pseudoconvex real hypersurface in $\mathbb{C}^n$, then
we say that $\Sigma$ is \emph{Bergman logarithmically flat} if there is a smoothly bounded, strongly pseudoconvex domain $\Omega \subset \mathbb{C}^n$ with $\Sigma \subset \partial\Omega$ such that the coefficient function $\psi$ in the Fefferman expansion~\eqref{Fefferman} of the Bergman kernel $K_{\Omega}$ vanishes to infinite order along $\Sigma$.
By the localization property of the Bergman kernel~\cite{Fe, BoSj, En01, HuLi23},
this condition is independent of the choice of the domain $\Omega$.
A strongly pseudoconvex real hypersurface $\Sigma$ in a complex manifold is said to be Bergman logarithmically flat if, for every $p \in \Sigma$, the germ of $\Sigma$ at $p$ is Bergman logarithmically flat in the above sense.

The prototype of a Bergman logarithmically flat hypersurface is the unit sphere, or more generally, any spherical hypersurface. By the latter we mean a hypersurface that is locally CR-diffeomorphic to a piece of the unit sphere. In the real three-dimensional case, i.e., for a real hypersurface $\Sigma$ in a complex manifold of dimension two, Graham--Burns \cite{Gra87b} and Boutet de Monvel \cite{Bou} showed that Bergman logarithmic flatness implies that $\Sigma$ must be spherical. In higher dimensions, however, this is no longer the case. Engli\v{s}--Zhang, in \cite{EnZh}, constructed Bergman logarithmically flat real hypersurfaces, even compact ones, that are not spherical. They showed that a certain unit circle bundle over any compact Hermitian symmetric space is Bergman logarithmically flat. This result yields examples of non-spherical Bergman logarithmically flat real hypersurfaces. The main goal of the present paper is to consider the more general situation of polarized unit circle bundles over a \k manifold $M$ and understand when this hypersurface is Bergman logarithmically flat in terms of the \k geometry of $M$. As a consequence, we also establish a direct generalization of the Engli\v{s}--Zhang result in \cite{EnZh} to the situation where $M$ is compact and locally homogeneous.

We shall now explain our results in more detail. We will primarily be interested in real hypersurfaces that arise as the unit circle bundle of a negative line bundle over a \k manifold. We first recall the notion of a polarized \k manifold.

\begin{defn}\label{Defn polarized manifold} {\rm
	If $(L, h) \to M$ is a positive Hermitian line bundle over a complex manifold $M$ and $g$ the K\"ahler metric on $M$ corresponding to the K\"ahler form
	\[
	\omega_g = \frac{1}{2} \, \mathrm{Ric}(L, h),
	\]
	then we say that $(L, h)$ \emph{polarizes} the K\"ahler manifold $(M, g)$ (or simply the \k metric $g$) and call the pair $(M, g; L, h)$ a \emph{polarized manifold}.}
\end{defn}

We emphasize that, in our definition, $(M, g)$ is not required to be compact or complete, in contrast to some conventions in the literature.
	
Now, let $(L, h)$ be a positive line bundle polarizing a K\"ahler manifold $(M, g)$ of complex dimension $n$.
Let $(L^*, h^* = h^{-1})$ be the dual bundle of $(L, h)$ and denote by
\[
D=D(L^*) := \{ v \in L^* \;:\; |v|_{h^*} < 1 \},
\quad
S=S(L^*) := \{ v \in L^* \;:\; |v|_{h^*} = 1 \}
\]
the disk and circle bundles of $L^*$, respectively, where $|v|_{h^*}$ denotes the norm of $v$ with respect to $h^*$.
It is well known that the negativity of $(L^*, h^*)$ implies that $S$ is strongly pseudoconvex.
Our goal, as stated above, is to characterize Bergman logarithmic flatness of $S$ in terms of the K\"ahler geometry of $(M, g)$.
This problem is ultimately related to the Bergman kernel of the disk bundle $D$.
Therefore, we shall first recall the definition, following Kobayashi \cite{Ko}, of the Bergman kernel of $D$.
Denote by $L^2_{(n+1, 0)}(D)$ the space of $L^2$ integrable $(n+1, 0)$ forms on $D$  equipped with the following inner product
\begin{equation*}
	(f, g)_{L^2_{(n+1, 0)}(D)}=(\sqrt{-1})^{(n+1)^2} \int_D f\wedge \oo{g}, \quad \forall f, g\in L^2_{(n+1, 0)}(D),
\end{equation*}
and define the Bergman space of $D$ to be
\begin{equation*}
	A^2_{(n+1, 0)}(D):=\left\{f\in L^2_{(n+1, 0)}(D): f \text{ is a holomorphic $(n+1, 0)$ form on } D \right\}.
\end{equation*}
The Bergman space $A^2_{(n+1, 0)}(D)$ is a separable Hilbert space. We shall assume that $A^2_{(n+1, 0)}(D)$ is not trivial. Taking any orthonormal basis $\{\varphi_k\}_{k=1}^d$ of $A^2_{(n+1, 0)}(D)$ (here $1\leq d\leq \infty$), we define the \emph{Bergman kernel form} of $D$ to be
\begin{equation*}
	K_D(z, w)=(\sqrt{-1})^{(n+1)^2} \sum_{k=1}^d \varphi_k(z) \wedge \oo{\varphi_k(w)}, \quad \forall z, w \in D.
\end{equation*}
Then, $K_D=K_D(z, z)$ is a real-valued, real-analytic $(n+1, n+1)$ form on $D$, which is independent of the choice of orthonormal basis.

It is natural to ask if $K_D$ admits a Fefferman type asymptotic expansion, similar to \eqref{Fefferman}, as $x$ approaches $S$, the topological boundary of $D$ in $L^*$. Corollary \ref{Cor Fefferman expansion} below yields an affirmative answer to this question, provided that the underlying complex manifold $M$ admits a complete \k metric $\widehat{g}$. We stress that this metric $\widehat{g}$ need not be the metric $g$ polarized by $(L, h)$. In particular, the \k manifold $(M,g)$ is not required to be complete. 
We begin by establishing a localization result for the Bergman kernel $K_D$.

\begin{thm}\label{Thm localization}
Let $(M, g; L,h)$ be a polarized manifold, and let $(L^*, h^*)$ be the dual line bundle of $(L,h).$ Assume that $M$ admits some complete \k metric.
Let $\Omega$ be a relatively compact pseudoconvex domain in $L^*$ with smooth, strongly pseudoconvex boundary such that $\Omega\subseteq D$ and there exist $p\in S$ and a small neighborhood $W$ of $p$ in $L^*$ such that $W\cap \Omega=W\cap D$. Then, the difference of Bergman kernel forms $K_D-K_{\Omega}$ in $W\cap \Omega$ extends as a $C^{\infty}$-smooth form on $W\cap \oo{\Omega}$.
\end{thm}


We remark that if $M$ is compact, the conclusion in Theorem \ref{Thm localization} also follows from earlier localization results in \cite{Fe, BoSj, HuLi23}. These results, however, do not apply when $M$ is noncompact.

Using Theorem \ref{Thm localization}, we shall show that $K_D$
admits a Fefferman type asymptotic expansion. In order to do this, we shall define a canonical $(n+1, n+1)$-form associated to $(L^*, h^*)$. In what follows and throughout the paper, we shall say that $(U, z)$ is a trivializing coordinate chart of $(M, g; L,h)$ if $(U,z)$ is a coordinate chart of $M$ and there exists a frame $e_L$ of $L$ (equivalently, a frame $e_{L^*}$ of $L^*$) over $U$. Let $(M, g; L,h)$ be a polarized manifold, $(L^*, h^*)$ the dual line bundle of $(L,h)$, and $(U, z)$ a trivializing coordinate chart of $(M, g; L, h)$ with a frame $e_{L^*}$ of $L^*$ over $U$, so that
\begin{equation*}
	\pi^{-1}(U)=\left\{ \xi e_{L^*}(z): (z, \xi)\in U\times \mathbb{C} \right\}.
\end{equation*}
Let $e^{\phi(z)}=h^*(e_{L^*}(z), e_{L^*}(z) )$ on $U$, and recall that $g(z)= (g_{i\bar{j}}(z))_{1\leq i, j\leq n}$ is the metric associated to the \k form $\omega_g=\frac{1}{2}\Ric(L, h)$ on $U$ (i.e., $g_{i\oo{j}}=\partial_{z_i}\partial_{\oo{z_j}}\phi$ and $\omega_g=\frac{\sqrt{-1}}{2}\partial\dbar\phi$). Introduce the locally defined $(n+1, n+1)$-form
\begin{equation}\label{canonical top form}
	T(z, \xi)=(\sqrt{-1})^{(n+1)^2}  e^{\phi(z)} \det g(z) ~dz\wedge d\xi\wedge d\oo{z}\wedge d\oo{\xi}.
\end{equation}
We shall verify in Section \ref{Sec 2} (see Proposition \ref{prop_T}) that the definition of $T$ is independent of the choices of local coordinate chart $(U, z)$ and the frame $e_{L^*}$. As a result, $T$ given by \eqref{canonical top form} is a well-defined global smooth $(n+1, n+1)$-form on $L^*$.

\begin{defn}\label{Defn canonical form} {\rm The globally defined $(n+1, n+1)$-form $T$ on $L^*$ defined locally by \eqref{canonical top form} is called \emph{the canonical top form of $(L^*, h^*)$}.}
\end{defn}

Using the notation just introduced, we now state the following Fefferman type expansion for the Bergman kernel form $K_D$ of the disk bundle $D$.
In what follows, $\oo{D}=D\cup S$ denotes the topological closure of $D$ in $L^*.$

\begin{cor}\label{Cor Fefferman expansion}
Let $(M, g; L,h)$ be a polarized manifold, and let $(L^*, h^*)$ be the dual line bundle of $(L,h).$ Assume that $M$ admits some complete \k metric.
Let $\rho(v)=1-|v|_{h^*}^2$ for $v\in L^*$ and denote by $T$ the canonical top form of $(L^*, h^*).$ Then, the Bergman kernel form $K_D$ of $D$ satisfies
	\begin{equation}\label{Fefferman expansion for the disk bundle}
		K_D=\Bigl(\frac{\Phi}{\rho^{n+2}}+\Psi\log\rho\Bigr)T \quad \text{ on } D
	\end{equation}
	for some functions $\Phi, \Psi \in C^{\infty} (\oo{D})$.
	
Furthermore, the choices of $\Phi$ and $\Psi$ in \eqref{Fefferman expansion for the disk bundle} are unique, up to the addition of a smooth form in $C^{\infty}(\oo{D})$ that vanishes to infinite order along $S$ (i.e., $O(\rho^{\infty})$ along $S$). Moreover, $\Psi$ vanishes to infinite order along an open piece $\Sigma\subset S$ (i.e., $\Psi=O(\rho^{\infty})$ along $\Sigma$) if and only if $\Sigma$ is Bergman logarithmically flat.
\end{cor}



In view of Corollary \ref{Cor Fefferman expansion}, to study Bergman logarithmic flatness of $S$, it is equivalent to investigate when $\Psi$ vanishes to infinite order along $S$. The main point here is the localization in Theorem \ref{Thm localization}, since the definition of Bergman logarithmic flatness is local and involves the Bergman kernel of a domain fully contained within a coordinate chart at a boundary point, whereas the coefficient function $\Psi$ is determined by the global Bergman kernel form and so is more closely related to the \k geometry of $M$.

For the purpose of studying $\Psi$, we will review and utilize another notion of a Bergman kernel in the setting of a polarized manifold. Let $(M, g; L, h)$ be a polarized manifold and $(E, h_E)$ a Hermitian vector bundle over $M$ with a Hermitian metric $h_E$. Set $E_k=L^k\otimes E$ for $k\in \mathbb{Z}^+$. There is a naturally induced metric $h_k=h^k\otimes h_E$ on $E_k$. The corresponding Bergman space $A^2(M, E_k)$ consists of $L^2$-integrable holomorphic sections of $E_k$ with respect to the following inner product:
\begin{equation}\label{L2 inner product}
	(s, s')_{L^2}=\int_M h_k\left(s(z), s'(z)\right) dV_g, \quad \forall s, s' \in A^2(M, E_k).
\end{equation}
Here $dV_g=\frac{1}{n!}\omega_g^n$ is the volume form of $g$, with $\omega_g$ as defined in Definition \ref{Defn polarized manifold}. We note that $A^2(M, E_k)$ is a Hilbert space, and as before, we shall assume that $A^2(M, E_k)\neq \{0\}$ and let $\{s_{k, 1}, \cdots, s_{k, d}\}$ be an orthonormal basis of $A^2(M, E_k)$ with respect to the above inner product, where $1\leq d=d(k)\leq \infty$. The Bergman kernel of $(E_k, h_k)$ is defined as the following section of $E_k\otimes \oo{E_k}$:
\begin{equation*}
	K_{E_k}(z, w)=\sum_{j=1}^d s_{k, j}(z)\otimes \oo{s_{k, j}(w)}.
\end{equation*}
The \emph{Bergman kernel function} of $(E_k, h_k)$ is defined to be the pointwise norm of $K_{E_k}(z, z)$,
\begin{equation*}
	B_{E_k}(z):=\left|K_{E_k}(z, z)\right|_{h_k}=\sum_{j=1}^d\bigl|s_{k,j}\bigr|^2_{h_k}.
\end{equation*}
The study of the asymptotic expansion of $B_{E_k}(z)$ as $k \to \infty$ is a classical and important problem in complex analysis and geometry. This line of research goes back to the seminal works \cite{Ti90, Ze98, Ca99, Lu00} on the Tian--Yau--Zelditch--Catlin expansion, where the primary case of interest was for compact manifolds $M$. Their results have motivated extensive and far-reaching developments. Engli\v{s} \cite{En02} established an asymptotic expansion in the case where $M$ is a pseudoconvex domain in Euclidean space, using a Laplace integral method. Dai--Liu--Ma \cite{DLM04, DLM06} obtained an asymptotic expansion on compact symplectic manifolds using the heat kernel.
Berman, Berndtsson and Sj\"ostrand \cite{BBSj08} gave a direct approach to the Tian--Yau--Zelditch--Catlin expansion by constructing a local Bergman kernel and exploiting its local reproducing property. This method was further developed by Deleporte--Hitrik--Sj\"ostrand \cite{DHSj22} and Hitrik--Stone \cite{HS22}. Note that \cite{BBSj08} treated the case where $M$ is a compact complex manifold, while \cite{DHSj22} and \cite{HS22} considered the case where $M$ is a pseudoconvex domains in $\mathbb{C}^n$.
In \cite{HsMa14}, Hsiao--Marinescu studied Bergman kernels over general (not necessarily compact) complex manifolds. They established asymptotics for the Bergman kernel defined with respect to the $L^2$ inner product \eqref{L2 inner product}, where $g$ is replaced by another complete metric that is not necessarily polarized by $(L,h)$.
Further work can also be found in \cite{Wa03, ShZe02, MM08, LiuLu11, Xu12, MaSa24} and references therein. Many researchers have contributed to this topic, and the list here is by no means intended to be complete.
The reader is referred to the monograph by Ma and Marinescu \cite{MM07} for a comprehensive exposition of the asymptotic expansion of the Bergman kernel and its applications. 

The particular case of interest to us is when $E$ is the canonical line bundle $\mathcal{C}_M$ of $(M, g)$, equipped with the natural Hermitian metric induced by $g$. This case is critical for our investigation, as it is intimately linked to the Bergman kernel form of $D$. We emphasize that, for our purposes, it is natural to use the metric $g$ polarized by $(L,h)$ (rather than other Kähler metrics on $M$) to define both the $L^2$ inner product \eqref{L2 inner product} on the Bergman space and the Hermitian metric on $\mathcal{C}_M$, since our goal is to study Bergman logarithmic flatness of $S$ in terms of the Kähler geometry of the induced metric $g$.
We shall need a Tian--Yau--Zelditch--Catlin type expansion of the Bergman kernel of $L^k \otimes \mathcal{C}_M$, as stated in Proposition~\ref{Prop TYZ expansion} below. This result should be more or less known to experts. 
For the convenience of the reader and the sake of being self-contained, we sketch a proof in the appendix by following the framework of \cite{BBSj08} (and also \cite{DHSj22, HS22}) with minor adjustments. In fact, in the appendix we will establish a more general result, which will be used in our subsequent paper on the Szeg\H o kernel \cite{EXX25}.

In the following result, we shall again assume that $M$ admits some complete \k  metric, and this metric is not necessarily the metric $g$ polarized by $(L, h)$.

\begin{prop}\label{Prop TYZ expansion}
Let $(M, g; L,h)$ be a polarized manifold, and assume that $M$ admits some complete \k metric. Denote by $H=(\det g)^{-1}$ the Hermitian metric on the canonical line bundle $\mathcal{C}_M$ of $M$ induced by $g$. Consider the Bergman kernel function $B_k=B_{E_k}$ of $(E_k=L^k\otimes \mathcal{C}_M, h_k=h^k\otimes H),$ where $k\in \mathbb{Z}^+$. Then, there exist unique (real-valued) smooth functions $a_j$, for $j\geq 0$, such that $B_k$ has the following asymptotic expansion:
	\begin{equation}\label{TYZ expansion intro}
		B_k(z)\sim \left(\frac{k}{\pi}\right)^n\sum_{j=0}^{\infty} \frac{a_j(z)}{k^j}, ~\text{ locally uniformly for } z\in M.
	\end{equation}
Moreover, $a_0=1$ and the functions $a_{j}$, for $j\geq 1$, are universal polynomials in the curvature of $(M, g)$ and its covariant derivatives.
\end{prop}

\begin{rmk}\label{rmk asymptotic expansion meaning}
	The asymptotic expansion \eqref{TYZ expansion intro} means that, for any $N\geq 0, m\geq 0$ and $X \subset\subset M$, there exists $C_{N, m, X}$ depending on $N, m, X$, $(M, g)$ and $(L, h),$ while independent of $k$, such that
	\begin{equation*}
		\Bigl\|B_k-\bigg(\frac{k}{\pi}\bigg)^n\sum_{j=0}^{N} \frac{a_j}{k^j} \Bigr\|_{C^m(X)}\leq C_{N, m, X} k^{n-N-1}, \quad \forall k\geq 1.
	\end{equation*}
\end{rmk}

\begin{rmk}\label{rmk asymptotic expansion application} A crucial point in Proposition \ref{Prop TYZ expansion} is that the $a_j$ are {\em universal} polynomials in the curvature and its covariant derivatives. Thus, if $M'\subset M$ also admits a complete \k metric and $B'_k$ denotes the corresponding Bergman kernel function of the polarized manifold $(M', g|_{M'}; L|_{M'},h|_{M'})$, then $B'_k$, which can be significantly different from $B_k$ for fixed $k$, satisfies the same asymptotic expansion as $k\to \infty$. This is an important ingredient in the proof of Theorem \ref{Thm localization}.
\end{rmk}

\begin{rmk}\label{rmk aj curvature}
Following the formulas in \cite{Wang05, Wa03}, we can compute explicit expressions for $a_0, a_1$ and $a_2$. In particular, we have $a_0=1$ (as already mentioned in Proposition \ref{Prop TYZ expansion}),  $a_1=-\frac{R}{2}$, and 
\begin{equation*}
	a_2=-\frac{1}{6}\Delta_g R+\frac{1}{24}\left(|\Rm|^2-4|\Ric|^2+3R^2\right),
\end{equation*}
where $\Rm, \Ric$ and $R$ are respectively the Riemannian, Ricci, and scalar curvatures of $(M, g)$, and $\Delta_g$ is the Laplacian with respect to $g$.
We will not use these explicit formulas in our proofs.
\end{rmk}

The asymptotic expansions of the Bergman kernels $K_D$ and $B_k$, established in Corollary \ref{Cor Fefferman expansion} and Proposition \ref{Prop TYZ expansion}, are intimately related to each other. Indeed, by knowing one of them, we can recover the other.
We will demonstrate this in Theorem \ref{Thm the coeffiecients relation} below. In its statement, as in Corollary \ref{Cor Fefferman expansion}, we denote by $\rho$ the natural defining function of the circle bundle $S$ in $L^*: \rho(v)=1-|v|_{h^*}^2$ for $v\in L^*$. We denote by $\pi: L^*\rightarrow M$ the natural projection from $L^*$ to $M$.

\begin{thm}\label{Thm the coeffiecients relation}
Let $(M, g; L,h)$ be a polarized manifold, and let  $(L^*, h^*)$ be the dual line bundle of $(L, h)$. Assume that $M$ admits some complete metric.
Let $\Phi, \Psi$ be the coefficient functions appearing in the expansion \eqref{Fefferman expansion for the disk bundle} of the Bergman kernel form $K_D$.
Let $\{a_j\}_{j=0}^{\infty}$ be the coefficient functions appearing in the expansion \eqref{TYZ expansion intro} of the Bergman kernel function $B_k$.
 Then 
there exist two unique families of smooth functions $\{\alpha_j\}_{j=0}^{n+1}$ and $\{\beta_j\}_{j=0}^{\infty}$ on $M$ such that the following hold in $D$ near $S$,

\begin{align}
	\Phi&=\sum_{j=0}^{n+1} (\alpha_j \circ \pi) \rho^j +O(\rho^{n+2}) \label{relation between Phi and A}.\\
	\Psi&=\sum_{j=0}^N (\beta_j \circ \pi) \rho^j+O(\rho^{N+1}), ~\text{for every}~N \geq 0 \label{relation between Psi and B}.
\end{align}
Moreover, each $\alpha_j$, $0\leq j\leq n+1$, can be explicitly computed in terms of $a_0, \ldots, a_j$ by equation \eqref{relation between A and a}, and each $\beta_j$, $j\geq 0$, can be explicitly computed in terms of $a_{n+2}, \ldots, a_{n+2+j}$ by equation \eqref{relation between B and a}.
As a consequence, for any $U\subset M$ and any $j\geq 0$, the following are equivalent:
\begin{itemize}
\item [{\rm (i)}] $\beta_0=\ldots=\beta_j=0$ on $U$;
\item [{\rm (ii)}] $a_{n+2}=\ldots=a_{n+2+j}=0$ on $U$.
\end{itemize}	
\end{thm}

\begin{rmk}
	Equation \eqref{relation between Phi and A} means that there exists some $(n+1, n+1)$-form $\Phi_{n+2}\in C^{\infty}_{(n+1, n+1)}(\oo{D})$ such that
	\begin{equation*}
		\Phi=\sum_{j=0}^{n+1} (\alpha_j \circ \pi) \rho^j T+\rho^{n+2}\Phi_{n+2} \quad \text{ on } D.
	\end{equation*}
	Similarly, equation \eqref{relation between Psi and B} means that, for each $N\geq 0$, there exists some $(n+1, n+1)$-form $\Psi_{N+1}\in C^{\infty}_{(n+1, n+1)}(\oo{D})$ such that
	\begin{equation*}
		\Psi=\sum_{j=0}^{N} (\beta_j \circ \pi) \rho^jT+\rho^{N+1}\Psi_{N+1} \quad \text{ on } D.
	\end{equation*}
\end{rmk}

\bigskip

Theorem \ref{Thm the coeffiecients relation} allows us to characterize Bergman logarithmic flatness of  the circle bundle $S$ in terms of the \k geometry of $(M, g)$.
\begin{thm}\label{Thm BLF criterion}
Let $(M, g; L,h)$ be a polarized manifold, and let $S$ be the circle bundle of the dual bundle $(L^*, h^*)$. Assume that $M$ admits some complete \k metric. Let $\{a_j\}_{j=0}^{\infty}$ be
the coefficient functions appearing in the expansion \eqref{TYZ expansion intro} of the Bergman kernel function $B_k$. Let $U$ be an open subset of $M$ and let $\Sigma$ be defined by $\Sigma=\pi^{-1}(U)\cap S$. Then the following are equivalent:
\begin{itemize}	
\item[{\rm (i)}] $a_{n+2+m}= 0$ on $U$ for all $m\geq 0$;
\item[{\rm (ii)}]  $\Sigma$ is Bergman logarithmically flat.
\end{itemize}
\end{thm}

Since all the coefficients $a_j$ are polynomials in the curvature of $(M, g)$ and its covariant derivatives, Theorem \ref{Thm BLF criterion} allows us to characterize Bergman logarithmically flatness of $\Sigma$ in terms of the \k geometry of $(M, g).$ We remark that Lu and Tian, in their seminal work \cite{LT04}, initiated the study of characterizing the logarithmic singularity in terms of the coefficients of the Tian--Yau--Zelditch--Catlin expansion. Moreover, when $M$ is compact and $\Sigma=S$, the implication ``(ii) $\implies$ (i)'' in Theorem \ref{Thm BLF criterion} was already proved by Lu and Tian  (see \cite[Theorem 2.4]{LT04}). They posed the converse implication as an open question (\cite{LT04}, and \cite{Lu-PC}).  Inspired by their work, we prove the following more decisive result in the compact case.

\begin{cor}\label{Cor BLF criterion compact}
Let $(M, g; L,h)$ be a polarized compact manifold, and let $S$ be the circle bundle of the dual bundle $(L^*, h^*)$. Let $\{a_j\}_{j=0}^{\infty}$ be
the coefficient functions appearing in the expansion \eqref{TYZ expansion intro} of the Bergman kernel function $B_k$. Then, the following are equivalent:
\begin{itemize}	
\item[{\rm (i)}] For each $m\geq 0$, $a_{n+2+m}$ is a constant function on $M$;
\item[{\rm (ii)}] $S$ is Bergman logarithmically flat;
\item[{\rm (iii)}] For each $m \geq 0$, $a_{n+2+m}$ is either non-negative on $M$ or non-positive on $M$.
\end{itemize}

\end{cor}

The problem of finding non-spherical Bergman logarithmically flat hypersurfaces, of dimension $5$ and higher, has attracted the attention of many researchers and is closely related to the Ramadanov conjecture (see \cite{EX24} and references therein). Pioneering work was carried out by Engli\v{s}--Zhang in their paper \cite{EnZh}, which has influenced much subsequent research along this line; see also \cite{LMZ17, X23, EX24} for more recent developments. We conclude this introduction with an application of Corollary \ref{Cor BLF criterion compact}, which directly generalizes the work in \cite{EnZh}.

By Corollary~\ref{Cor BLF criterion compact}, constructing a Bergman logarithmically flat circle bundle $S$ over a compact manifold $M$ amounts to finding a polarized compact manifold $(M, g; L, h)$ such that the associated functions $a_{n+2+m}$ are constant on $M$, for all $m \geq 0$. We note that, e.g., local homogeneity of $(M, g)$ is sufficient to guarantee this. Recall that a Kähler manifold $(M, g)$ is locally homogeneous if, for every pair of points $p, q \in M$, there exist neighborhoods $V_p$ and $V_q$ of $p$ and $q$, respectively, and a biholomorphism $f: V_p \to V_q$ such that $f$ preserves the metric: $f^*g = g$ on $V_p$.

\begin{cor}\label{Cor cpt homogeneous is BLF}
	Let $(M, g; L, h)$ be a polarized manifold. Assume that $(M, g)$ is compact and locally homogeneous. Then, the circle bundle $S$ of $(L^*, h^*)$ is Bergman logarithmically flat. Furthermore, $S$ is spherical 
	 if and only if $(M,g)$ is locally holomorphically isometric to one of the following complex space forms:
	\begin{itemize}
		\item [(1)] $(\mathbb{B}^n, \lambda \,\omega_{-1})$ for some $\lambda\in \mathbb{R}^+$,
		\item [(2)] $(\mathbb{CP}^n, \lambda\, \omega_1)$ for some $\lambda\in \mathbb{R}^+$,
		\item [(3)] $(\mathbb{C}^n, \omega_0)$,
		\item [(4)] $(\mathbb{B}^l\times \mathbb{CP}^{n-l}, \lambda \omega_{-1}\times \lambda\omega_1)$ for some $1\leq l\leq n-1$ and some $\lambda\in\mathbb{R}^+$.
	\end{itemize}
	Here $\omega_c$ in the statement of the corollary denotes the \k metric with constant holomorphic sectional curvature $c$.
\end{cor}

Corollary~\ref{Cor cpt homogeneous is BLF} allows for the construction of many examples of non-spherical Bergman logarithmically flat circle bundles. We remark, as mentioned above, that when $(M, g)$ is a Hermitian symmetric space of compact type, Engli\v{s}--Zhang \cite{EnZh} had already shown that $S$ is Bergman logarithmically flat. They further proved that $S$ is spherical if and only if $M$ is the projective space $\mathbb{CP}^n$, thereby providing the first examples of non-spherical Bergman logarithmically flat hypersurfaces. See also the subsequent work of Loi--Mossa--Zuddas \cite{LMZ17}, which treated the case where $(M, g)$ is a compact, simply-connected homogeneous Kähler--Einstein manifold of classical type.

It is natural to ask whether the compactness condition in Corollary~\ref{Cor BLF criterion compact} can be removed while still concluding the Bergman logarithmic flatness of $S$. To this end, we propose the following conjecture and will study it in the non-compact manifold case in a separate article.
\begin{conj}
	Let $(M, g; L, h)$ be a polarized manifold. If $(M, g)$ is locally homogeneous, then the circle bundle $S$ of $(L^*, h^*)$ is Bergman logarithmically flat.
\end{conj}


The paper is organized as follows. In Section \ref{Sec 2}, we first prove Theorem \ref{Thm localization} and Corollary \ref{Cor Fefferman expansion}. Section \ref{Sec 3} is devoted to the proof of Theorem \ref{Thm the coeffiecients relation}. Then in Section 4, we establish Theorem \ref{Thm BLF criterion}, Corollary \ref{Cor BLF criterion compact}, and Corollary \ref{Cor cpt homogeneous is BLF}. Finally, in Appendix \ref{Sec Appendix}, we prove a more general version of Proposition \ref{Prop TYZ expansion}


\section{Proof of Theorem \ref{Thm localization} and Corollary \ref{Cor Fefferman expansion}}\label{Sec 2}
In this section, we first verify the the fact that the canonical top form $T$ of $(L^*, h^*)$ in Definition \ref{Defn canonical form}, i.e., the $(n+1,n+1)$-form given by \eqref{canonical top form}, is well-defined. After that, we will prove
Theorem \ref{Thm localization} and Corollary \ref{Cor Fefferman expansion}.


Let $(M, g; L, h)$ be a polarized manifold, and $(L^*, h^*)$ be the dual line bundle of $(L,h).$  We
shall verify that the definition of the canonical top form $T$, given by \eqref{canonical top form},
is independent of the choices of local coordinate chart $(U, z)$ and the frame $e_{L^*}$.
For this purpose, we pick any trivializing coordinate chart $(U, z)$ of $(M, g; L, h)$ together with a local frame $e_{L^*}$ of $L^*$. Thus, in the notation in Definition \ref{Defn canonical form},
\begin{equation}\label{canonical top form 1}
	T=(\sqrt{-1})^{(n+1)^2}  e^{\phi(z)} \det g(z) ~dz\wedge d\xi\wedge d\oo{z}\wedge d\oo{\xi}.
\end{equation}
Suppose $(\widetilde{U}, \widetilde{z})$ is another trivializing coordinate chart with a local frame $\widetilde{e}_{L^*}$ of $L^*$.  It induces another coordinate system $(\widetilde{z}, \widetilde{\xi})$ for $L^*$ over $\widetilde{U}.$  Over $U\cap\widetilde{U}$,  we have the transition map $\widetilde{z}=F(z)$ for some holomorphic map $F$, and $\widetilde{\xi}=\eta(z)\xi$ for some holomorphic nonvanishing function $\eta$. Writing $e^{\widetilde{\phi}(\widetilde{z})}=h^*(\widetilde{e}_{L^*}(\widetilde{z}), \widetilde{e}_{L^*}(\widetilde{z}) )$ on $\widetilde{U}$, and $g(\widetilde{z})= (\widetilde{g}_{k\oo{l}}(\widetilde{z}))_{1\leq k, l\leq n}$ for the metric under the new coordinates, then at every $ p \in \pi^{-1}(U\cap\widetilde{U})$,
\begin{align*}
	\det g=&\det \bigl(g_{i\oo{j}}\bigr)=\det \biggl(\widetilde{g}_{k\oo{l}}\frac{\partial \widetilde{z_k}}{\partial z_i} \oo{\frac{\partial \widetilde{z_l}}{\partial z_j}}\biggr)=\det \widetilde{g}~ |\det JF|^2,\\
	e^{\phi(z)}=&h^*\bigl(e_{L^*}(z), e_{L^*}(z)\bigr)=h^*\bigl(\eta(z)\widetilde{e}_{L^*}(\widetilde{z}), \eta(z)\widetilde{e}_{L^*}(\widetilde{z})\bigr)=|\eta(z)|^2 e^{\widetilde{\phi}(\widetilde{z})}.
\end{align*}
By noting that at $p$,
\begin{align*}
	d\widetilde{z}\wedge d\widetilde{\xi} \wedge d\oo{\widetilde{z}}\wedge d\oo{\widetilde{\xi}}=|\det JF|^2|\eta|^2 dz\wedge d\xi \wedge d\oo{z}\wedge d\oo{\xi},
\end{align*}
we conclude that the definitions of $T$ coincide in the two coordinate systems $(z, \xi)$ and $(\widetilde{z}, \widetilde{\xi})$, and hence $T$ is a well-defined global $(n+1,n+1)$-form on $L^*$, as claimed. We record this result as:

\begin{prop}\label{prop_T}
The canonical top form $T$, given by \eqref{canonical top form}, is independent of the choice of local coordinate chart $(U, z)$ and frame $e_{L^*}$. Hence, $T$ is a globally defined $(n+1, n+1)$-form on $L^*$.
\end{prop}

We next prove Theorem \ref{Thm localization} and Corollary \ref{Cor Fefferman expansion}. For that, we need the following relation between the two types of Bergman kernels. 
\begin{thm}\label{Thm relation between two BKs}
Let $(X, g; L, h)$ be a polarized manifold, and $(L^*, h^*)$ be the dual line bundle of $(L,h).$  Let $T$ denote the canonical top form of $(L^*,h^*)$ and
	\begin{equation*}
		D=D(L^*)=\{v\in L^*: \rho(v):=1-|v|^2_{h^*}>0\}
	\end{equation*}
the disk bundle of  $(L^*, h^*)$. Let $\mathcal{C}_X$ denote the canonical line bundle of $(X,g)$ equipped with the Hermitian metric $H=(\det g)^{-1}$. Then, the Bergman kernel form $K_D$ of $D$ and the Bergman kernel function $B_{k+1}$ of the line bundle $L^{k+1}\otimes \mathcal{C}_X$ has the following relation:
	\begin{equation}\label{relation between two BKs}
		K_D=\Bigl(\sum_{k=0}^{\infty} \frac{k+1}{2^{n+1}\pi} (1-\rho)^k B_{k+1}\Bigr)T.
	\end{equation}
The infinite series above converges uniformly on every compact subset of $D$.
\end{thm}

\begin{proof}
Recall that the Bergman space of $D$ is defined by
\begin{equation*}
	A^2_{(n+1, 0)}(D):=\left\{ \sigma \text{ is a holomorphic $(n+1, 0)$ form with }~  (\sqrt{-1})^{(n+1)^2}\int_D \sigma\wedge\oo{\sigma}<\infty \right\}.
\end{equation*}
In what follows, we will omit the subscript $(n+1, 0)$ and simply write  $A^2(D)$ for brevity.
We define an $S^1$-action $r_{\theta}$ on $L^*$ as follows. Choose any trivializing coordinate chart $(U, z)$ of $(X, g; L, h)$ together with a frame $e_{L^*}$ of $L^*$ over $U$. This induces a trivialization of $L^*$ over $U$:
\begin{equation}\label{local trivialization of L*}
	\pi^{-1}(U) \ni \xi e_{L^*}(z)\rightarrow (z, \xi) \in U\times \mathbb{C}.
\end{equation}
Then, over $\pi^{-1}(U)$, $r_{\theta}$ is defined by  $r_{\theta}(z, \xi)=(z, e^{i\theta}\xi)$. It is clear that the definition of $r_{\theta}$ does not depend on the choice of local trivialization of $L^*$. Note also that $r_{\theta}$ restricts to a fiber-length preserving biholomorphism of $D$. 
Moreover, $r_{\theta}$ induces a  pullback operator $r_\theta^*$ acting on the  differential forms on $D$, and we have
\begin{equation*}
	\int_D \sigma\wedge\oo{\sigma}=\int_D r_\theta^* \Bigl(\sigma\wedge\oo{\sigma}\Bigr)= \int_D \Bigl( r_{\theta}^*\sigma \wedge \oo{r_{\theta}^*\sigma} \Bigr), \quad \forall \sigma\in A^2(D).
\end{equation*}
Therefore,
\begin{equation}\label{S1 action preserves the L2 norm}
	\|r_{\theta}^*(\sigma)\|_{L^2(D)}=\|\sigma\|_{L^2(D)}, \quad \forall \sigma\in A^2(D).
\end{equation}
Consequently, the map $\mu: \theta \in \mathbb{T} \rightarrow r_{\theta}^*$ gives a unitary representation of the torus $\mathbb{T}$ on the Hilbert space $A^2(D).$


Next, for any $m\in \mathbb{Z}$ we define
\begin{equation*}
	A^2_{m}(D)=\{\sigma\in A^2(D): r_{\theta}^*\sigma=e^{\sqrt{-1}m\theta}\sigma ~\text{ for all } \theta \in \mathbb{T} \}.
\end{equation*}
Since $r_{\theta}$ preserves the $L^2$ norm on $D$, each $A^2_{m}(D)$ is a closed subspace of $A^2(D)$. Let $\sigma \in  A^2(D).$ Under a trivialization $(z, \xi)$ of $L^*$ over some open set $U \subset X$, since $\pi^{-1}(U) \cap D$ is complete circular in $\xi$, we may write
\begin{equation}\label{Taylor expansion in a disk}
	\sigma=\sum_{m=0}^{\infty} b_m(z)\xi^m dz\wedge d\xi \quad \text{on } \pi^{-1}(U) \cap D,
\end{equation}
for some holomorphic functions $b_m(z)$ on $U$. It follows readily that $\sigma\in A^2_{m+1}(D)$ if and only if $b_j(z)=0$ for all $j\neq m$, i.e., when $\sigma=b_m(z)\xi^m dz\wedge d\xi$.
Furthermore, one can easily verify that when $ m \leq 0$, $A^2_m(D)$ is always trivial; and if $m_1 \neq m_2,$ then $A^2_{m_1}(D)  \perp A^2_{m_2}(D)$ in $A^2(D)$.

Recall that $\mu$ defines a unitary representation of the 1-torus $\mathbb{T}$ on $A^2(D)$.
By classical harmonic/Fourier analysis (e.g., Stone's theorem and the spectral theorem), $\mu$ decomposes into an orthogonal direct sum of (irreducible)
one-dimensional unitary representations of $\mathbb{T}$ of the form
\[
\theta \in \mathbb{T} \;\longmapsto\; e^{\sqrt{-1}m\theta}, \qquad m \in \mathbb{Z}.
\]
Consequently, $A^2(D)$ decomposes into an orthogonal direct sum of one-dimensional subspaces $H_k$,
indexed by $k \in \mathbb{Z}$, and each $H_k$ is contained in some $A^2_m(D)$. Using this, we can
establish the following proposition, which states that $A^2(D)$ decomposes into an orthogonal direct sum of the subspaces $A^2_{m+1}(D)$. Since the proof is standard, we omit it here.

\begin{prop}\label{Prop orthogonal decomposition of A2}
We have the orthogonal decomposition
\begin{equation*}
		A^2(D)=\bigoplus_{m=0}^{\infty}{}^{\perp} A^2_{m+1}(D).
\end{equation*}
\end{prop}
	


Recall that the Bergman space $A^2(X, L^{m+1} \otimes \mathcal{C}_X)$ consists of $L^2$-integrable holomorphic sections of $L^{m+1} \otimes \mathcal{C}_X$.
Next we will prove that this Bergman space is isomorphic to $A^2_{m+1}(D)$ as a Hilbert space. 
\begin{lemma}\label{Lem Hilbert space isomorphism}
	There exists an isomorphism between Hilbert spaces $I: A^2_{m+1}(D)\rightarrow A^2(X, L^{m+1} \otimes \mathcal{C}_X)$ given by the following. Let $\sigma_m \in A^2_{m+1}(D)$ and $(z, \xi)$ be a trivialization of $L^*$ over $U\subset X$ with $e_{L^*}$ (resp. $e_L$) the local frame of $L^*$ (resp. $L$) over $U$. If we write $\sigma_m=b(z)\xi^m dz\wedge d\xi$ on $\pi^{-1}(U) \cap D$, then $I$ is defined by
	\begin{equation*}
		I(\sigma_m)=\sqrt{\frac{2^{n+1}\pi}{m+1}}b(z) e_L^{m+1} \otimes dz \quad \text{over } U.
	\end{equation*}
\end{lemma}

\begin{proof}[Proof of Lemma $\ref{Lem Hilbert space isomorphism}$]
	We first prove that $I$ is well-defined. For that, we need to verify that the definition of $I$ is independent of the choice of the coordinate chart and trivialization of $L^*$.
	Let $(\widetilde{z}, \widetilde{\xi})$ be another such trivialization of $L^*$ over $V$ with $V\cap U\neq \emptyset$. Under the new coordinate system, we write $\sigma_m=\widetilde{b}(\widetilde{z}) \widetilde{\xi}^m d\widetilde{z}\wedge d\widetilde{\xi}$. Over $U \cap V$, we have the relation $\widetilde{z}=F(z)$ for some holomorphic map $F$, and $\widetilde{e}_{L^*}(\widetilde{z})=\frac{1}{\eta(z)}e_{L^*}(z)$ for some holomorphic nonvanishing function $\eta$, so that $\widetilde{\xi}=\eta(z)\xi$ and $\widetilde{e}_{L}(\widetilde{z})=\eta(z) e_{L}(z)$. Using this relation, we see that, over $U \cap V$,
	\begin{equation*}
		\sigma_m=\widetilde{b}(\widetilde{z}) \widetilde{\xi}^m d\widetilde{z}\wedge d\widetilde{\xi}=\widetilde{b}(F(z)) ~\eta(z)^{m+1} \xi^m\det JF~ dz\wedge d\xi.
	\end{equation*}
	By comparing it with $\sigma_m=b(z)\xi^m dz\wedge d\xi$, we obtain $\widetilde{b}(F(z))\eta(z)^{m+1}\det JF=b(z)$. Therefore,
	\begin{equation*}
		\widetilde{b}(\widetilde{z}) \widetilde{e}_L^{m+1} \otimes d\widetilde{z} =\widetilde{b}(F(z))  \det JF ~\eta(z)^{m+1} e_L^{m+1} \otimes dz=b(z)  e_L^{m+1} \otimes dz.
	\end{equation*}
Consequently, the image $I(\sigma_m)$ is the same in both coordinate systems, and thus it gives a well-defined global holomorphic section of $L^{m+1} \otimes \mathcal{C}_X$.
	
It is easy to that see $I$ is linear. We next prove $I(\sigma_m) \in A^2(X,  L^{m+1} \otimes \mathcal{C}_X)$ and $I$ is unitary. Fix $\sigma_m, \tau_m\in A^2_{m+1}(D)$. Pick a countable, locally finite open cover $\{U_j\}$ of $X$, such that each $U_j$ is contained in some trivializing coordinate chart. Let $\{\rho_j\}$ be a partition of unity subordinate to $\{U_j\}$. Let $(z, \xi)=(z_j, \xi_j)$ be the coordinates obtained from the trivialization of $L^*$ over $U_j$ as in \eqref{local trivialization of L*}. Over $U_j$, we write $\tau_m=d(z)\xi^m dz\wedge d\xi$ and thus $I(\tau_m)=\sqrt{\frac{2^{n+1}\pi}{m+1}} d(z) e_L^{m+1} \otimes dz$. Likewise, we write $\sigma_m$ and $I(\sigma_m)$ as in Lemma \ref{Lem Hilbert space isomorphism} in the coordinate system $(z, \xi).$ Then we have
	\begin{align*}
		\bigl(I(\sigma_m), I(\tau_m)\bigr)_{L^2(X, L^{m+1} \otimes \mathcal{C}_X)}=\sum_{j}\frac{2^{n+1}\pi}{m+1} \int_{U_j}\rho_j b(z)\oo{d(z)} h^{m+1}~ \frac{(\sqrt{-1})^{n^2}}{2^n} dz\wedge d\oo{z}.
	\end{align*}
	On the other hand, we also have
	\begin{align*}
		\bigl(\sigma_m, \tau_m\bigr)_{A^2(D)}=&\sum_{j} (\sqrt{-1})^{(n+1)^2} \int_{D \cap  \pi^{-1}(U_j)} \rho_j b(z) \oo{d(z)} ~|\xi|^{2m} dz\wedge d\xi \wedge d\oo{z} \wedge d\oo{\xi}\\
		=& \sum_{j} \int_{U_j}\rho_j b(z)\oo{d(z)}~\Bigl(\int_{|\xi|^2<h(z, \oo{z})}  |\xi|^{2m} \sqrt{-1}~d\xi \wedge d\oo{\xi}\Bigr) ~(\sqrt{-1})^{n^2}dz\wedge d\oo{z}\\
		=&\sum_{j}\frac{2\pi}{m+1} \int_{U_j}\rho_j b(z)\oo{d(z)} h^{m+1}~ (\sqrt{-1})^{n^2} dz\wedge d\oo{z}.
	\end{align*}
	Therefore,
	\begin{equation*}
		\bigl(\sigma_m, \tau_m\bigr)_{A^2(D)}=\bigl(I(\sigma_m), I(\tau_m)\bigr)_{L^2(X, L^{m+1} \otimes \mathcal{C}_X)}.
	\end{equation*}
Taking $\tau_m=\sigma_m$, we see that $I(\sigma_m) \in A^2(X,  L^{m+1} \otimes \mathcal{C}_X)$; and the last equality shows $I$ is unitary.

Finally to prove the surjectivity of $I$, we define another operator $J:~~ A^2(X, L^{m+1} \otimes \mathcal{C}_X) \rightarrow A^2_{m+1}(D)$ as follows. Let $s_m \in A^2(X,  L^{m+1} \otimes \mathcal{C}_X)$ and let $(z, \xi)$ be a trivialization of $L^*$ over $U\subset X$ with $e_{L^*}$ (resp. $e_L$) the local frame of $L^*$ (resp. $L$). Under the coordinate system, write $s_m=\gamma(z) e_L^{m+1} \otimes dz$ over $U,$ and define $J(s_m)=\sqrt{\frac{m+1}{2^{n+1}\pi}} \gamma(z)\xi^m dz\wedge d\xi $ over $\pi^{-1}(U)$. Similarly as above, one can verify that $J$ is a well-defined linear unitary operator from  $A^2(X, L^{m+1} \otimes \mathcal{C}_X)$ to $A^2_{m+1}(D)$. Moreover, it is easy to see $J$ is the inverse of $I$. Hence $I$ is surjective, and this proves the lemma.
\end{proof}

We are in a position to prove \eqref{relation between two BKs}. Let $\{s_{k+1, \alpha}\}_{\alpha}$ be an orthonormal basis of $A^2(X,  L^{k+1} \otimes \mathcal{C}_X)$. Then the Bergman kernel of
$(L^{k+1} \otimes \mathcal{C}_X, h^{k+1} \otimes H)$ is given by
\begin{equation*}
	K_{k+1}=\sum_{\alpha} s_{k+1, \alpha}\otimes \oo{s_{k+1, \alpha}}.
\end{equation*}
In a local trivialization $(z, \xi)$ of $L^*$ over $U$ as in \eqref{local trivialization of L*}, we can write $s_{k+1, \alpha}=\sqrt{\frac{2^{n+1}\pi}{k+1}} b_{k+1, \alpha} dz\otimes e_L^{k+1}$ for some holomorphic function $b_{k+1, \alpha}$. Then the Bergman kernel function of $L^{k+1} \otimes \mathcal{C}_X$ is
\begin{equation}\label{Bergman function in local coordiantes}
	B_{k+1}:=|K_{k+1}|_{L^{k+1}\otimes \mathcal{C}_X} =\sum_{\alpha}\frac{2^{n+1}\pi}{k+1} |b_{k+1, \alpha}|^2 (\det g)^{-1} h^{k+1}.
\end{equation}
From Proposition \ref{Prop orthogonal decomposition of A2} and Lemma \ref{Lem Hilbert space isomorphism}, it follows that $\{I^{-1}(s_{k+1, \alpha})\}_{k, \alpha}$ forms an orthonormal basis of $A^2(D)$. Therefore, over $\pi^{-1}(U)$ we have
\begin{align*}
	K_D=&(\sqrt{-1})^{(n+1)^2}\sum_{k=0}^{\infty}\sum_{\alpha} I^{-1}(s_{k+1, \alpha})\wedge \oo{I^{-1}(s_{k+1, \alpha})}\\
	=&(\sqrt{-1})^{(n+1)^2}\sum_{k=0}^{\infty}\sum_{\alpha} |b_{k+1, \alpha}|^2 |\xi|^{2k}dz\wedge d\xi \wedge d\oo{z} \wedge d\oo{\xi}.
\end{align*}
By \eqref{Bergman function in local coordiantes}, we can further simplify it into
\begin{align*}
	K_D=&\Bigl(\sum_{k=0}^{\infty} \frac{k+1}{2^{n+1}\pi}B_{k+1} |\xi|^{2k} h^{-k}\Bigr) (\sqrt{-1})^{(n+1)^2} h^{-1} \det g ~ dz\wedge d\xi \wedge d\oo{z} \wedge d\oo{\xi}\\
	=& \Bigl( \sum_{k=0}^{\infty} \frac{k+1}{2^{n+1}\pi} B_{k+1} (1-\rho)^k \Bigr) T.
\end{align*}
Here we have used the definition \eqref{canonical top form} of $T$ and the fact that $\rho=1-|\xi|^2h^{-1}$ in the last equality. Therefore, \eqref{relation between two BKs} is verified.

To prove the uniform convergence, recall that on any compact set $Y\subset X$, there exists some constant $C_Y>0$, independent of $k$, such that
\begin{equation}\label{eqn bk cy}
		B_{k}=|K_{k}|_{L^{k}\otimes \mathcal{C}_X} \leq C_Y k^n.
\end{equation}
This follows from the extremal characterization of the Bergman function and the sub-mean value inequality for holomorphic sections (see \cite[Equation (2.5)]{Ber03} and \cite[Lemma 4.1]{HKSX16}).
Fix any compact subset $Z$ of $D$. Then, writing $\pi: (L^*, h^*)\rightarrow X$ for the natural projection, $Y=\pi(Z)$ is compact in $X$, and thus \eqref{eqn bk cy} holds.
Moreover, there exists $r_0\in \mathbb{R}$ such that $0\leq 1-\rho<r_0<1$ on $Z$. Then on $Z$ the summand in the infinite sum of \eqref{relation between two BKs} is less than
	\begin{equation*}
	C_Y \frac{k+1}{2^{n+1}\pi} ~(r_0)^k ~(k+1)^n,
	\end{equation*}
which proves the uniform convergence of the infinite sum over $Z$, as desired. This finishes the proof of Theorem \ref{Thm relation between two BKs}.
\end{proof}

Before proving Theorem \ref{Thm localization}, we need another preliminary result. Let $(M, g; L,h)$ and $(L^*, h^*)$ be as in Theorem \ref{Thm localization}. Let $(U, z)$ be a trivializing coordinate chart, inducing a coordinate system $(z, \xi)$ of $L^*$ over $U$.
Let $\mathcal{B} \subset\subset U$ be an Euclidean ball in $(U, z)$, which in particular admits a complete \k metric. Write $\pi: L^*\rightarrow M$ for the natural projection from $L^*$ to $M$. Write
\begin{equation*}
	D_{\mathcal{B}}=D\cap \pi^{-1}(\mathcal{B}), \quad\quad S_{\mathcal{B}}=S\cap \pi^{-1}(\mathcal{B}).
\end{equation*}
We then show that the Bergman kernel forms $K_{D_{\mathcal{B}}}$ and $K_D$ of $D_{\mathcal{B}}$ and $D$ satisfy the following:
\begin{prop}\label{Prop localize BK to small disk bundle}
	The difference $K_D-K_{D_{\mathcal{B}}}$ extends as a $C^{\infty}$-smooth $(n+1, n+1)$-form on $D_{\mathcal{B}}\cup S_{\mathcal{B}}$.
\end{prop}
\begin{proof}
Write $\mathcal{C}_M$ and $\mathcal{C}_{\mathcal{B}}$ for the canonical line bundles of $(M,g)$ and $(\mathcal{B},g)$ respectively, both equipped with the Hermitian metric $H=(\det g)^{-1}$.
Denote by $K_{k+1}$ and $\widehat{K}_{k+1}$ the Bergman kernels of the line bundles of $L^{k+1}\otimes \mathcal{C}_M$ and $(L|_{\mathcal{B}})^{k+1}\otimes \mathcal{C}_{\mathcal{B}}$, respectively. 
We apply Theorem \ref{Thm relation between two BKs} to $(X=M, g; L,h)$ and $(X=\mathcal{B},g; L,h)$ respectively, to obtain
	\begin{align}
		K_D&=\Bigl(\sum_{k=0}^{\infty} \frac{k+1}{2^{n+1}\pi} (1-\rho)^k B_{k+1}\Bigr)T \label{K on D},\\
		K_{D_{\mathcal{B}}}&=\Bigl(\sum_{k=0}^{\infty} \frac{k+1}{2^{n+1}\pi} (1-\rho)^k \widehat{B}_{k+1}\Bigr)T \label{K on D_B}.
	\end{align}
Fix any $p\in S_{\mathcal{B}}$. Recall that both $M$ and $\mathcal{B}$ admit complete \k metrics. Let the $a_j$ be as in Proposition \ref{Prop TYZ expansion} and note that the $a_j$ corresponding to $\widehat{B}_{k}$ are the restrictions of the $a_j$ for $B_k$ restricted to $\mathcal{B}$ (see Remark \ref{rmk asymptotic expansion application}). For $N\geq 0$, let $A_k^{(N)} :=\bigl(\frac{k}{\pi}\bigr)^n\sum_{j=0}^{N} \frac{a_j(z)}{k^j}.$   By Proposition \ref{Prop TYZ expansion} (see Remark \ref{rmk asymptotic expansion meaning}), we can pick open sets $V\subset\subset \mathcal{B}$  with $\pi(p)\in V$ such that for every integer $m \geq 1$,
	\begin{align*}
		\left\| B_k- A_k^{(N)}\right\|_{C^m(V)}\leq \frac{C_{N, m}}{k^{N+1-n}}, \quad
		\left\| \widehat{B}_{k}- A_k^{(N)} \right\|_{C^m(V)}\leq \frac{C_{N, m}}{k^{N+1-n}},
	\end{align*}
	for some constant $C_{N, m}>0$ depending on $N, m, V$ and $(L, h), M, \mathcal{B}$, while independent of $k$. Consequently,
	\begin{equation}\label{difference of two Bergman functions}
		\left\| B_{k+1}- \widehat{B}_{k+1} \right\|_{C^m(V)}\leq \frac{2C_{N, m}}{(k+1)^{N+1-n}}, \quad\text{for every } N\geq n.
	\end{equation}
	For simplicity, we denote $J_{k+1}= B_{k+1}- \widehat{B}_{k+1}$. Let $\alpha=(\alpha_1, \alpha_2, \alpha_3, \alpha_4) \in \mathbb{N}^n \times \mathbb{N}^n \times \mathbb{N} \times \mathbb{N}$ be a multi-index and define the differential operator $\mathcal{L}^{\alpha}$ on $\pi^{-1}(V)$ by
$$
\mathcal{L}^{\alpha}=\frac{\partial^{|\alpha|}}{\partial z^{\alpha_1} \partial \oo{z}^{\alpha_2} \partial \xi^{\alpha_3} \partial \oo{\xi}^{\alpha_4}},
$$
in the local coordinate system $(z, \xi)$ of $L^*$ over $U$.	
Equation \eqref{difference of two Bergman functions} implies that
	\begin{equation}\label{J_{k+1} derivatives}
	\left| \mathcal{L}^{\alpha} \left((1-\rho)^k J_{k+1}\right)\right| \leq \frac{C_{N, \alpha} ~k^{|\alpha|}}{(k+1)^{N+1-n}} \quad \text{on } (D_{\mathcal{B}}\cup S_{\mathcal{B}})\cap \pi^{-1}(V) \quad\text{for every } N\geq n.
	\end{equation}
	where $C_{N, \alpha}$ is some constant depending on $N, \alpha, V$ and $(L, h), M, \mathcal{B}$, independent of $k$.
	
	Also, by \eqref{K on D} and \eqref{K on D_B}, we have
	\begin{equation}\label{K on D and K on D_B difference}
		K_D-K_{D_{\mathcal{B}}}=\Bigl(\sum_{k=0}^{\infty} \frac{k+1}{2^{n+1}\pi} (1-\rho)^k J_{k+1}\Bigr)T=:Q \cdot T \quad\text{on } D_{\mathcal{B}}.
	\end{equation}
Here, $Q$ is the function given by the infinite sum in the parentheses above. We next show

\medskip

\textbf{Claim.} The difference $K_D-K_{D_{\mathcal{B}}}$ extends $C^{\infty}$-smoothly across $S_{\mathcal{B}} \cap \pi^{-1}(V)$, and in particular across $p$.

\medskip
	
{\em Proof of Claim.} Let $m \geq 1$ and $\mathcal{L}^{\alpha}$ be the differential operator in $\pi^{-1}(V)$ as above with $|\alpha| \leq m$.
By \eqref{J_{k+1} derivatives}, we see that for any $N \geq n,$ there exists some constant $C_{N, m}$  independent of $k$ such that the following holds on $(D_{\mathcal{B}}\cup S_{\mathcal{B}})\cap \pi^{-1}(V):$
	\begin{equation*}
		\sum_{k=0}^{\infty} \frac{k+1}{2^{n+1}\pi} \left| \mathcal{L}^{\alpha} \left((1-\rho)^k J_{k+1}\right)\right|\leq \frac{C_{N, m}}{2^{n+1}\pi} \sum_{k=0}^{\infty} \frac{1}{(k+1)^{N-n-m}}, \quad \forall |\alpha| \leq m.
	\end{equation*}
If we choose $N=m+n+2$, then the above converges uniformly on $(D_{\mathcal{B}}\cup S_{\mathcal{B}})\cap \pi^{-1}(V)$. Since $m$ is arbitrary, by \eqref{K on D and K on D_B difference}, for each differential operator $\mathcal{L}^{\alpha}$ with $|\alpha|\geq 0$, $\mathcal{L}^{\alpha} Q$ extends to a continuous function on $(D_B\cup S_{\mathcal{B}})\cap \pi^{-1}(V)$. Then by a classical extension theorem \cite{Se64} (see also \cite{Wh34} and \cite[Chapter VI, Theorem 5]{St70}), $Q$ and thus $K_D-K_{D_\mathcal{B}}$ extend $C^{\infty}$-smoothly across $S_{\mathcal{B}} \cap \pi^{-1}(V)$, in particular across $p$. \qed
	
Since $p\in S_{\mathcal{B}}$ is arbitrary, we arrive at the desired conclusion of Proposition \ref{Prop localize BK to small disk bundle}.
\end{proof}

We are now ready to prove Theorem \ref{Thm localization}.
\begin{proof}[Proof of Theorem $\ref{Thm localization}$]
	It suffices to show $K_D-K_{\Omega}$ extends smoothly across every $q\in W\cap \partial\Omega$. Fix any such point $q$. Choose a smaller smoothly bounded strongly pseudoconvex domain $\Omega_0\subset \Omega$ with $q\in \partial\Omega_0$ such that
	\begin{itemize}
		\item[(a)] $\Omega_0 \subset\subset \pi^{-1}(\mathcal{B})$, for some small Euclidean ball $\mathcal{B} \subset\subset U$ in some trivializing coordinate chart $(U, z)$ of $(M, g; L,h)$.
		\item[(b)] There exists some small neighborhood $W_0\subset W$ of $q$ in $L^*$ satisfying
		\begin{equation*}
			W_0\cap \Omega_0=W_0\cap\Omega=W_0\cap D.
		\end{equation*}
	\end{itemize}
Since $\Omega$ is relatively compact in $L^*$, by the localization of the Bergman kernels on strongly pseudoconvex domains (see \cite{Fe, BoSj}, and \cite[Proposition 3.1]{HuLi23}), the difference of the Bergman kernel forms $K_{\Omega}-K_{\Omega_0}$ extends across $W_0\cap \partial \Omega_0$. Therefore, to show $K_D-K_{\Omega}$ extends smoothly across $q\in W_0\cap \partial\Omega_0$, it suffices to prove the same holds for $K_D-K_{\Omega_0}$. For that, we will first need the following simple fact.

\bigskip
\textbf{Claim.} The domain $D_{\mathcal{B}}=D\cap \pi^{-1}(\mathcal{B})$ is pseudoconvex.
\begin{proof}[Proof of Claim]
Write $\mathcal{B}=\{|z-z_0|<R\}$ for some $z_0 \in U$ and $R>0$, in the trivializing coordinate chart $(U, z)$. Then in the local coordinates $(z, \xi)$,  we have $D_{\mathcal{B}}$ is defined by
\begin{equation*}
D_{\mathcal{B}}=\{(z, \xi): |z-z_0|<R, ~\rho(z, \xi):=1-|\xi|^2e^{\phi}>0 \};
\end{equation*}
and $\max\{-\log(R^2-|z-z_0|^2), -\log \rho \}$ is a plurisubharmonic exhaustion function on $D_{\mathcal{B}}$.
\end{proof}
	Now, note that $D_{\mathcal{B}}$ is a bounded pseudoconvex domain in $\mathbb{C}^{n+1}$ and $q$ is a smooth strongly pseudoconvex boundary point of $D_{\mathcal{B}}$.  We apply Theorem 4.2 in \cite{En01} to conclude that $K_{D_{\mathcal{B}}}- K_{\Omega_0}$ extends $C^{\infty}$-smoothly across $q$. Combining this with Proposition \ref{Prop localize BK to small disk bundle}, we see $K_D-K_{\Omega_0}=(K_D-K_{D_{\mathcal{B}}})+(K_{D_{\mathcal{B}}}-K_{\Omega_0})$ extends $C^{\infty}$-smoothly across $q$, as desired.
\end{proof}

We end this section with the proof of Corollary \ref{Cor Fefferman expansion}.
\begin{proof}[Proof of Corollary $\ref{Cor Fefferman expansion}$]
For any $p\in S$, pick a small smoothly bounded strongly pseudoconvex domain $\Omega\subset\subset L^*$,  such that $p\in \partial\Omega$ and the following conditions hold:
\begin{itemize}
	\item [(a)] $\Omega \subset\subset \pi^{-1}(U)$ for some trivializing coordinate chart $(U, z)$ of $(M, g; L,h)$.
	\item[(b)] There exists some small neighborhood $W_p \subset\subset \pi^{-1}(U)$ of $p$ in $L^*$ such that $W_p \cap \Omega=W_p \cap D$.
\end{itemize}
	
Recall that $K_D$ and $K_{\Omega}$ denote the Bergman kernel forms of $D$ and $\Omega$, respectively. By Theorem \ref{Thm localization}, the difference of the Bergman kernel forms $K_{D}-K_{\Omega}$ is a smooth $(n+1, n+1)$-form in a neighborhood of $W_p \cap \oo{\Omega}$. Thus, denoting by $T$ the canonical top form of $(L^*, h^*),$
\begin{equation}\label{eq difference H}
		\widehat{K}:=\frac{K_{D}-K_{\Omega}}{T}
\end{equation}
is a smooth function in a neighborhood of $W_p \cap \oo{\Omega}$.  Since $\Omega$ can be regarded as a bounded strongly pseudoconvex domain in a complex Euclidean space, by Fefferman \cite{Fe},
	\begin{equation}\label{eq K_Omega expansion}
		K_{\Omega}=\left(\frac{\phi_{\Omega}}{\rho_{\Omega}^{n+2}}+\psi_{\Omega}\log\rho_{\Omega}\right)T
	\end{equation}
for some smooth functions $\phi_{\Omega}$ and $\psi_{\Omega}$ in a neighborhood of $\oo{\Omega}$. Shrinking $W_p$ as needed, we can assume
	\begin{equation}\label{eq defining function modified}
		\rho_{\Omega}=\rho \quad \text{in } W_p.
	\end{equation}
Then equations \eqref{eq difference H}, \eqref{eq K_Omega expansion} and \eqref{eq defining function modified} yield
	\begin{equation}\label{eq K_D in a small neighborhood}
		K_{D}=\left(\frac{\phi_{\Omega}+ \widehat{K} \rho^{n+2}}{\rho^{n+2}}+\psi_{\Omega}\log\rho\right)T=:\left(\frac{\phi_p}{\rho^{n+2}}+\psi_p\log \rho\right)T \quad \text{in } W_p \cap \Omega.
	\end{equation}
It is clear that $\phi_p$ and $\psi_p$ are smooth functions on $W_p \cap \oo{\Omega}$. Note that all such open sets $\{W_p\}_{p\in S}$ form an open cover of $S$. We can then choose another open set $O$ such that $\oo{O}$ (the closure of $O$ in $L^*$) is contained in $D$,  and
	\begin{equation*}
		\oo{D} \subset \widetilde{W}:=O \cup \Bigl(\cup_{p\in S} W_p \Bigr).
	\end{equation*}
Write $\widehat{K}_D:=\frac{K_D}{T}$, which is a smooth function on $D$. Note on $O$ we have
	\begin{equation*}
		K_{D}=\left( \frac{\widehat{K}_{D}\rho^{n+2}}{\rho^{n+2}}+0\cdot\log \rho\right) T=:\left( \frac{\widehat{\phi}}{\rho^{n+2}}+\widehat{\psi}\log \rho \right)T.
	\end{equation*}
	Choose a smooth partition of unity $\{h_p\}_{p\in S} \cup \{\widehat{h}\}$ on $\widetilde{W}$ subordinate to the open cover $\{W_p, O\}_{p\in S}$, and set
	
	\begin{equation*}
		\Phi=\sum_{p\in S} h_p\phi_p+\widehat{h}\widehat{\phi}, \qquad \Psi=\sum_{p\in S} h_p\psi_p+\widehat{h}\widehat{\psi}.
	\end{equation*}
	Then we have
	\begin{equation}\label{eq K_D expansion}
		K_{D}=\left( \frac{\Phi}{\rho^{n+2}}+\Psi\log \rho \right) T  \quad \text{on } D.
	\end{equation}
	It is also clear that $\Phi$ and $\Psi$ are smooth functions on $\oo{D}$. This proves the first conclusion of Corollary \ref{Cor Fefferman expansion}.  Next we prove the uniqueness part of the choice of $\Phi$ and $\Psi$. Suppose
	\begin{equation*}
		K_{D}=\frac{\Phi}{\rho^{n+2}}+\Psi\log\rho=\frac{\widehat{\Phi}}{\rho^{n+2}}+\widehat{\Psi}\log\rho \quad \text{on } D.
	\end{equation*}
Here $\widehat{\Phi}$ and $\widehat{\Psi}$ are also smooth functions on $\oo{D}$. Then we have
	\begin{equation*}
		\Phi-\widehat{\Phi}+\bigl(\Psi-\widehat{\Psi}\bigr)\rho^{n+1}\log\rho=0 \quad \text{on } D.
	\end{equation*}
Applying \cite[Lemma 2.2]{FuWo} to the above equation along all smooth curves in $D$ intersecting transversally with $S$,
we see that $\Phi-\widehat{\Phi}$ and $\Psi-\widehat{\Psi}$ must vanish to infinite order along $S$.

To prove the last assertion of  Corollary \ref{Cor Fefferman expansion}, fix any $p\in \Sigma$. As above, choose a small neighborhood $W=W_p$ of $p$ in $L^*$, and a small, smoothly bounded, strongly pseudoconvex domain $\Omega \subset\subset L^*$ with $p\in \partial\Omega,$ such that
	\begin{equation*}
		W\cap \Omega= W\cap D, \qquad W\cap \partial\Omega=W\cap S \subset \Sigma.
	\end{equation*}
	Again, write $K_{\Omega}$ for the Bergman kernel of $\Omega$. By Theorem \ref{Thm localization}, we can assume \eqref{eq difference H}--\eqref{eq K_D in a small neighborhood} hold. By \eqref{eq K_D in a small neighborhood} and \eqref{eq K_D expansion}, the following holds in $W\cap D$:
	\begin{equation*}
		\bigl(\Phi-\phi_{\Omega}-\widehat{K}\rho^{n+2}\bigr)+\bigl(\Psi-\psi_{\Omega}\bigr)\rho^{n+2}\log\rho=0.
	\end{equation*}
By the above equation and \cite[Lemma 2.2]{FuWo}, $\Psi$ vanishes to infinite order along $\Sigma$ near $p$ if and only if $\psi_{\Omega}$ also vanishes to infinite order along  $\Sigma$ near $p$. The later, by definition, is equivalent to that $\Sigma$ is Bergman logarithmically flat at $p$. Since $p$ is arbitrary, we conclude
$\Psi$ vanishes to infinite order along  $\Sigma$ if and only if  $\Sigma$ is Bergman logarithmically flat. Hence we arrive at the desired conclusion in Corollary \ref{Cor Fefferman expansion}.
\end{proof}

\section{Proof of Theorem \ref{Thm the coeffiecients relation}}\label{Sec 3}

In this section, we prove Theorem \ref{Thm the coeffiecients relation}. The proof is inspired by \cite{EnZh} and \cite{Ca99}.
\begin{proof}[Proof of Theorem $\ref{Thm the coeffiecients relation}$]
First by Theorem \ref{Thm relation between two BKs}, we have
	\begin{equation}\label{relation between two BKs 2}
		K_D=\Bigl(\sum_{k=0}^{\infty} \frac{k+1}{2^{n+1}\pi} (1-\rho)^k B_{k+1}\Bigr)T.
	\end{equation}
	Here $B_{k+1}=\bigl|K_{k+1}\bigr|_{L^{k+1}\otimes \mathcal{C}_M}$ is the Bergman kernel function for the line bundle $L^{k+1}\otimes \mathcal{C}_M$. By Proposition \ref{Prop TYZ expansion} we have
	\begin{equation*}
		B_k(z)\sim \left(\frac{k}{\pi}\right)^n\sum_{j=0}^{\infty} \frac{a_j(z)}{k^j} \quad \mbox{ locally uniformly on } M,
	\end{equation*}
where the meaning of the above asymptotic expansion is explained in Remark \ref{rmk asymptotic expansion meaning}.
To make use of this expansion, we introduce the following definition and proposition.
	\begin{defn}\label{defn  O function}{\rm
		Let $F$ be a function on $\mathbb{Z}^{+}\times M$, smooth in the second variable. Let $q$ be a nonnegative integer. We say $F\in O_M(q)$ if, for any $m\geq 1$ and any compact subset $X\subset M$, there exists a constant $C>0$, independent of $k$, such that
		\begin{equation*}
			\bigl\| F(k, \cdot) \bigr\|_{C^m(X)}\leq \frac{C}{k^q}, \quad \forall k\geq 1.
		\end{equation*}
}
\end{defn}
Observe that the above inequality implies that, for any smooth differential operator $P$ of order $\leq m$ on $M$, there exists a constant $C$, depending on $P$ but independent of $k$, such that
		\begin{equation*}
			\bigl| P \left(F(k, \cdot) \right) \bigr|\leq \frac{C}{k^q} \text{ on } X, \quad \forall k\geq 1.
		\end{equation*}
It is clear that if $F\in O_M(p)$ and $p>q \geq 0$, then $F\in O_M(q)$. In the following,  as before, we denote by $\oo{D}=D \cup S$.

\begin{prop}\label{Prop sum of errors} The following hold:
\begin{itemize}
			\item[(1)] If $F\in O_M(p)$ for some $p\geq 1$, then $(k+1)F\in O_M(p-1)$.
			\item[(2)] If $H\in O_M(p)$ for some $p\geq 2$, then $\sum_{k=0}^{\infty} (1-\rho)^k H(k, \cdot)\in C^{p-2}(\oo{D})$.
\end{itemize}
	\end{prop}

\begin{proof}
Since $(1)$ is trivial, we only prove $(2)$. Pick any trivializing coordinate chart $(U, z)$ of $(M, g; L,h)$, which induces a coordinate system $(z, \xi)$ of $L^*$ over $U$. Let $\alpha=(\alpha_1, \alpha_2, \alpha_3, \alpha_4) \in \mathbb{N}^n \times \mathbb{N}^n \times \mathbb{N} \times \mathbb{N}$ be a multi-index with $|\alpha| \leq p-2$. Define a differential operator $\mathcal{L}^{\alpha}$ on $\pi^{-1}(U)$ by
$$
\mathcal{L}^{\alpha}=\frac{\partial^{|\alpha|}}{\partial z^{\alpha_1} \partial \oo{z}^{\alpha_2} \partial \xi^{\alpha_3} \partial \oo{\xi}^{\alpha_4}},
$$
in the local coordinate system $(z, \xi)$ of $L^*$ over $U$. Fix any open subset $V  \subset\subset U.$ Since  $0\leq 1-\rho\leq 1$ on $\oo{D}$ and $H$ is in $O_M(p)$, 
there exists some constant $C$, independent of $k$, such that the following holds on $\oo{D} \cap \pi^{-1}(V):$
	\begin{equation*}
		\left|\mathcal{L}^{\alpha} \left((1-\rho)^k H(k, \cdot) \right)\right| \leq \frac{Ck^{|\alpha|}}{k^p}\leq \frac{C}{k^2}, \quad \forall |\alpha| \leq p-2.
	\end{equation*}
	Thus $\sum_{k=0}^{\infty} \mathcal{L}^{\alpha} \left((1-\rho)^k H(k, \cdot)\right)$ converges absolutely and uniformly on $\oo{D} \cap \pi^{-1}(V)$. Then by using the same argument as in the proof of Proposition \ref{Prop localize BK to small disk bundle} and a classical extension theorem \cite{Se64} (see  also  \cite{Wh34} and \cite[Chapter VI, Theorem 5]{St70}), we arrive at the desired conclusion.
\end{proof}

Let $N\geq 2$ and write $B_k^{(N)}:=(\frac{k}{\pi})^n\sum_{j=0}^{N+n}\frac{a_j(z)}{k^j}$.  By Remark \ref{rmk asymptotic expansion meaning}, we have
\begin{equation*}
	R_k^{(N)}:=B_k-B_k^{(N)}\in O_M(N+1),
\end{equation*}
for each $N\geq 0$. It is clear that we also have $R_{k+1}^{(N)}\in O_M(N+1)$. By part (1) of Proposition \ref{Prop sum of errors}, $(k+1) R_{k+1}^{(N)}\in O_M(N)$. By part (2) of Proposition \ref{Prop sum of errors}, we have
\begin{equation}\label{error term in K_D}
	R^{(N)}:=\sum_{k=0}^{\infty} \frac{k+1}{2^{n+1}\pi} (1-\rho)^k R_{k+1}^{(N)}\in C^{N-2}(\oo{D}).
\end{equation}
 By \eqref{relation between two BKs 2} and the definitions of $R_{k+1}^{(N)}$ and $R^{(N)}$, we deduce
\begin{align}\label{K_D in terms of main term and error term}
	\begin{split}
		K_D=&\Bigl(\sum_{k=0}^{\infty} \frac{k+1}{2^{n+1}\pi} (1-\rho)^k B_{k+1}^{(N)}+R^{(N)}\Bigr)T\\
		=&\left[ \sum_{k=0}^{\infty} \left( \frac{(k+1)^{n+1}}{(2\pi)^{n+1}} (1-\rho)^k \sum_{j=0}^{N+n}\frac{a_j(z)}{(k+1)^j} \right)+R^{(N)}\right]T=: \bigl[B^{(N)}+R^{(N)}\bigr] T.
	\end{split}
\end{align}
We then split the sum $\sum_{j=0}^{N+n}$ in $B^{(N)}$ into two parts $\sum_{j=0}^{n+1}$ and $\sum_{j=n+2}^{N+n}$, and obtain
\begin{align}\label{eqn bsn}
	B^{(N)}=\sum_{k=0}^{\infty} \left( \frac{(1-\rho)^k}{(2\pi)^{n+1}}\sum_{j=0}^{n+1} a_j (k+1)^{n+1-j} \right)+\sum_{k=0}^{\infty} \left( \frac{(1-\rho)^k}{(2\pi)^{n+1}} \sum_{j=n+2}^{N+n} \frac{a_j}{(k+1)^{j-n-1}} \right)=:(\mathrm{I})+(\mathrm{II}).
\end{align}
To deal with the terms (I) and (II),  we will need the following calculus fact, which was proved and used in \cite{Ca99} (see pp.\ 14-15 there). 
\begin{lemma}\label{Lem Taylor expansion}
	For $i\geq 0$, we have
	\begin{equation*}
		\frac{1}{(1-r)^{i+1}}=\sum_{l=0}^{\infty} \binom{l+i}{l} r^l, \quad \text{when } |r|<1.
	\end{equation*}
There exists a unique family of real numbers $\{Q_{jk}\}_{j, k=0}^{\infty}$ such that, for every $j\geq 0$, we have
	\begin{equation*}
		(1-r)^j\log(1-r)=\sum_{k=0}^{\infty} Q_{jk} r^k, \quad \text{when } |r|<1.
	\end{equation*}
In particular, $Q_{j0}=0$ for all $j;$ and  $Q_{jk}=\frac{(-1)^{j+1}j!}{k(k-1)\cdots (k-j)}$ for all $k>j\geq 0$.
\end{lemma}

\bigskip

We will treat terms (I) and (II) in the following Step 1 and Step 2 separately.

\medskip

\textbf{Step 1.} Reduction of (I).

Let us write
\begin{equation*}
	\binom{x+i}{i}=\frac{(x+i)\cdots (x+1)}{i!} \quad \text{ for a real variable $x$ and } i\in \mathbb{Z}^{+},
\end{equation*}
and by convention, $\binom{x}{0}=1.$ Note $\{(x+1)^{n+1-j}\}_{j=0}^{n+1}$ and $\{\binom{x+i}{i}\}_{i=0}^{n+1}$ both form a basis for the linear space of polynomials in $x$ of degree at most $n+1$. Let $\{\lambda_{ji}\}_{0\leq j, i\leq n+1}$ be the set of real numbers such that
\begin{equation*}
	(x+1)^{n+1-j}=\sum_{i=0}^{n+1}\lambda_{ji} \binom{x+i}{i}, \quad \forall\, 0\leq j\leq n+1.
\end{equation*}
Letting $x=k\geq 1$, we get
\begin{equation*}
	(k+1)^{n+1-j}=\sum_{i=0}^{n+1}\lambda_{ji} \binom{k+i}{k}.
\end{equation*}
Therefore,
\begin{align*}
	(\mathrm{I})=\sum_{k=0}^{\infty} \left( \frac{(1-\rho)^k}{(2\pi)^{n+1}}\sum_{j=0}^{n+1}a_j\sum_{i=0}^{n+1}\lambda_{ji}\binom{k+i}{k} \right)
	=\sum_{i=0}^{n+1}\Bigl(\sum_{j=0}^{n+1} \frac{\lambda_{ji}a_j}{(2\pi)^{n+1}} \Bigr)\sum_{k=0}^{\infty} \binom{k+i}{k}(1-\rho)^k.
\end{align*}
To simplify the above, write
\begin{equation}\label{relation between A and a}
	\alpha_{n+1-i}=\sum_{j=0}^{n+1} \frac{\lambda_{ji}a_j}{(2\pi)^{n+1}} \quad\text{for } 0\leq i\leq n+1.
\end{equation}
Since  $0\leq 1-\rho<1$ on $D$, we apply Lemma \ref{Lem Taylor expansion} with $r=1-\rho$ to get
\begin{equation}\label{eqn term I}
	(\mathrm{I})=\sum_{i=0}^{n+1}\alpha_{n+1-i}\frac{1}{\rho^{i+1}}=\rho^{-(n+2)}\sum_{j=0}^{n+1}\alpha_j\rho^j.
\end{equation}
This finishes the reduction of (I).

\textbf{Step 2.} Reduction of (II).

Let $\{Q_{jk}\}_{j,k=0}^{\infty}$ be as in Lemma \ref{Lem Taylor expansion}. 
We first prove the following proposition.
\begin{prop}\label{Prop write 1/k in terms of Q}
	For each $i\geq 0$, there exist constants $\{\tau_{ij}\}_{j=i}^{\infty}$ such that for any $N\geq i+2$ we have
	\begin{equation}\label{1/k in terms of Q}
		P_{N, k, i}:=\frac{1}{(k+1)^{i+1}}-\sum_{j=i}^{N-2}\tau_{ij}Q_{jk}=O\bigg(\frac{1}{(k+1)^N}\bigg), \quad \text{as } k\rightarrow \infty.
	\end{equation}
	Moreover, it holds that $\tau_{ii}=\frac{(-1)^{i+1}}{i!}$.
\end{prop}
\begin{proof}
Note that, by Lemma \ref{Lem Taylor expansion}, for each fixed $j$, we have
$$
Q_{jk}=\frac{(-1)^{j+1}j!}{(k+1)^{j+1}}+O\bigg(\frac{1}{(k+1)^{i+2}}\bigg), \quad \text{as}~k\rightarrow \infty.
$$
Thus by picking $\tau_{ii}=\frac{(-1)^{i+1}}{i!}$, we get
	\begin{equation*}
		\frac{1}{(k+1)^{i+1}}-\tau_{ii}Q_{ik}=O\bigg(\frac{1}{(k+1)^{i+2}}\bigg), \quad \text{as}~k\rightarrow \infty.
	\end{equation*}
A simple inductive argument shows the existence of the $\tau_{ij}$ for \eqref{1/k in terms of Q} to hold, and their uniqueness is also clear.
\end{proof}
Recall
\begin{align*}
	(\mathrm{II})=\sum_{k=0}^{\infty} \left(  \frac{(1-\rho)^k}{(2\pi)^{n+1}} \sum_{j=n+2}^{N+n}\frac{a_j}{(k+1)^{j-n-1}} \right)=\sum_{k=0}^{\infty} \left( \frac{(1-\rho)^k}{(2\pi)^{n+1}}\sum_{i=0}^{N-2}\frac{a_{n+2+i}}{(k+1)^{i+1}} \right).
\end{align*}
We use \eqref{1/k in terms of Q} to reduce the above to
\begin{align*}
	(\mathrm{II})=&\sum_{k=0}^{\infty} \left( \frac{(1-\rho)^k}{(2\pi)^{n+1}}\sum_{i=0}^{N-2} a_{n+2+i}\Bigl(P_{N, k, i} +\sum_{j=i}^{N-2}\tau_{ij}Q_{jk} \Bigr)  \right)\\
	=& \sum_{k=0}^{\infty} \left( \frac{(1-\rho)^k}{(2\pi)^{n+1}}\sum_{i=0}^{N-2} a_{n+2+i}P_{N, k, i} \right) +\sum_{k=0}^{\infty} \left( \frac{(1-\rho)^k}{(2\pi)^{n+1}}\sum_{i=0}^{N-2} \sum_{j=i}^{N-2}a_{n+2+i}\tau_{ij}Q_{jk} \right) =:(\mathrm{II}_a)+(\mathrm{II}_b).
\end{align*}
We first consider the term $(\mathrm{II}_a)$. Note $a_{n+2+i}$ is smooth on $M$, and by Proposition \ref{Prop write 1/k in terms of Q}, the constant $P_{N, k, i}\in O(\frac{1}{k^N})$. It then follows directly from Definition \ref{defn  O function} that
\begin{equation*}
	\sum_{i=0}^{N-2}a_{n+2+i}P_{N, k, i}\in O_M(N).
\end{equation*}
This yields $(\mathrm{II}_a)\in C^{N-2}(\oo{D})$, by part (2) of Proposition \ref{Prop sum of errors}. Next we consider term $(\mathrm{II}_b)$. On $D$, we have by the absolute convergence,
\begin{align*}
	(\mathrm{II}_b)=&\sum_{k=0}^{\infty} \left( \frac{(1-\rho)^k}{(2\pi)^{n+1}}\sum_{i=0}^{N-2}\sum_{j=i}^{N-2}a_{n+2+i}\tau_{ij}Q_{jk} \right)\\
	=&\sum_{i=0}^{N-2}\sum_{j=i}^{N-2} \left( \frac{a_{n+2+i}\tau_{ij}}{(2\pi)^{n+1}} \sum_{k=0}^{\infty} Q_{jk}(1-\rho)^k \right)=\sum_{j=0}^{N-2}\sum_{i=0}^{j} \left( \frac{a_{n+2+i}\tau_{ij}}{(2\pi)^{n+1}} \sum_{k=0}^{\infty} Q_{jk}(1-\rho)^k \right).
\end{align*}
Since  $0\leq 1-\rho<1$ on $D$, we apply Lemma \ref{Lem Taylor expansion} with $r=1-\rho$ to obtain
\begin{equation}\label{eqn iia f2}
	(\mathrm{II}_b)=\sum_{j=0}^{N-2}\sum_{i=0}^{j} \left(\frac{a_{n+2+i}\tau_{ij}}{(2\pi)^{n+1}} \rho^j\log\rho \right)=:\bigg( \sum_{j=0}^{N-2}\beta_j\rho^j \bigg) \log\rho,
\end{equation}
where
\begin{equation}\label{relation between B and a}
	\beta_j=\sum_{i=0}^j\frac{a_{n+2+i} \tau_{ij}}{(2\pi)^{n+1}} \quad \text{for } j\geq 0.
\end{equation}
To conclude Step 2,  we recall $(\mathrm{II})=(\mathrm{II}_a)+(\mathrm{II}_b)$. Then by \eqref{eqn iia f2} and the fact that $(\mathrm{II}_a)\in C^{N-2}(\oo{D})$,
\begin{equation}\label{eqn term II}
	(\mathrm{II})-\bigg(\sum_{j=0}^{N-2}\beta_j\rho^j \bigg) \log\rho\in C^{N-2}(\oo{D}).
\end{equation}

\bigskip


Recall by \eqref{eqn bsn}, $B^{(N)}=(\mathrm{I}) + (\mathrm{II}).$ Combining this with \eqref{eqn term I} and \eqref{eqn term II} in Step 1 and Step 2,   we obtain
\begin{equation*}
	B^{(N)}-\rho^{-(n+2)}\sum_{j=0}^{n+1}\alpha_j\rho^j-\bigg(\sum_{j=0}^{N-2}\beta_j\rho^j\bigg)\log\rho\in C^{N-2}(\oo{D}).
\end{equation*}
In view of the above, and \eqref{K_D in terms of main term and error term} and \eqref{error term in K_D}, we have
\begin{equation}\label{K_D expansion up to order N-2}
	K_D-\Biggl(\rho^{-(n+2)}\sum_{j=0}^{n+1}\alpha_j\rho^j+\Bigl(\sum_{j=0}^{N-2}\beta_j\rho^j\Bigr)\log\rho\Biggr)T\in C^{N-2}_{(n+1, n+1)}(\oo{D}).
\end{equation}
Here $C^{N-2}_{(n+1, n+1)}(\oo{D})$ denotes the spaces of $(n+1,n+1)$-forms on $\oo{D}$ of class $C^{N-2}.$
Recall by Corollary \ref{Cor Fefferman expansion},
\begin{equation}\label{Fefferman expansion for the disk bundle 2}
	K_D=\Bigl(\frac{\Phi}{\rho^{n+2}}+\Psi\log\rho\Bigr)T \quad \text{ on } D,
\end{equation}
for some functions $\Phi, \Psi \in C^{\infty} (\oo{D})$. We let
\begin{equation}\label{Phihat and Psihat def}
	\widehat{\Phi}:=\Phi-\sum_{j=0}^{n+1}\alpha_j\rho^j \in C^{\infty} (\oo{D}), \qquad \widehat{\Psi}:=\Psi-\sum_{j=0}^{N-2}\beta_j\rho^j \in C^{\infty} (\oo{D}).
\end{equation}
Combining \eqref{K_D expansion up to order N-2} and \eqref{Fefferman expansion for the disk bundle 2}, we obtain
\begin{equation*}
	f:=\frac{\widehat{\Phi}}{\rho^{n+2}}+\widehat{\Psi}\log\rho\in C^{N-2}(\oo{D}).
\end{equation*}
Thus, we have
\begin{equation}\label{identity on Phihat and Psihat}
	\widehat{\Phi}+\widehat{\Psi}\rho^{n+2}\log\rho+\rho^{n+2}f=0 \quad \text{on } D,
\end{equation}
where $f\in C^{N-2}(\oo{D})$.
We pause to prove the following lemma.
\begin{lemma}\label{Lem vanishing lemma}
	Assume that $\phi$ and $\psi$ are $C^{\infty}$-smooth functions on $(-\varepsilon, \varepsilon)$, for some $\varepsilon>0$, and $\varphi$ is a $C^m$-smooth function on $(-\varepsilon, \varepsilon)$, for some nonnegative integer $m$. If
	\begin{equation}\label{vanishing lemma}
		\phi+\psi t^{n+2}\log t+ t^{n+2}\varphi=0 \quad \text{on } (0, \varepsilon),
	\end{equation}
	then $\phi=O(t^{n+2})$ and $\psi=O(t^{m+1})$.
\end{lemma}
\begin{proof}
	We first show $\phi=O(t^{n+2})$. Suppose not. Then there exists $0\leq n_0\leq n+1$ and $a\neq 0$ such that $\phi(t)=at^{n_0}+O(t^{n_0+1})$. Dividing \eqref{vanishing lemma} by $t^{n_0}$, we get
	\begin{equation*}
		\frac{\phi}{t^{n_0}}+\psi t^{n+2-n_0}\log t+ t^{n+2-n_0}\varphi=0 \quad \text{on } (0, \varepsilon).
	\end{equation*}
	Let $t\rightarrow 0^+$ and we get $a=0$, an obvious contradiction. Therefore, $\phi=O(t^{n+2})$ and, consequently,
	\begin{equation*}
		\phi+t^{n+2}\varphi=t^{n+2}\widehat{\varphi}
	\end{equation*}
	for some $\widehat{\varphi}\in C^m(-\varepsilon, \varepsilon)$. We use this to reduce \eqref{vanishing lemma} to
	\begin{equation}\label{vanishing lemma 2}
		\psi\log t+\widehat{\varphi}=0 \quad \text{on } (0, \varepsilon).
	\end{equation}
	Next suppose $\psi\neq O(t^{m+1})$. Then there exists some $0\leq m_0\leq m$ and $b\neq 0$ such that
	\begin{equation*}
		\psi=bt^{m_0}+O(t^{m_0+1}).
	\end{equation*}
	On the other hand, since $\widehat{\varphi}\in C^m(-\varepsilon, \varepsilon)$, we have either $\widehat{\varphi}=O(t^{m_0})$ or $\widehat{\varphi}=ct^{\widehat{m}}+O(t^{\widehat{m}+1})$ for some $\widehat{m}<m_0$ and $c\neq 0$. If it is the former case, we divide \eqref{vanishing lemma 2} by $t^{m_0}\log t$ and let $t\rightarrow 0^+$ to obtain
	\begin{equation*}
		b=\lim_{t\rightarrow 0^+} \Bigl(\frac{\psi}{t^{m_0}}+\frac{\widehat{\varphi}}{t^{m_0}\log t}\Bigr)=0,
	\end{equation*}
	a contradiction. If it is the latter case, dividing \eqref{vanishing lemma 2} by $t^{\widehat{m}}$ and letting $t\rightarrow 0^+$, we see
	\begin{equation*}
		c=\lim_{t\rightarrow 0^+} \Bigl(\frac{\psi}{t^{\widehat{m}}}\log t+\frac{\widehat{\varphi}}{t^{\widehat{m}}}\Bigr)=0,
	\end{equation*}
	which is also a contradiction. This finishes the proof of Lemma \ref{Lem vanishing lemma}.
\end{proof}

Applying Lemma \ref{Lem vanishing lemma} to \eqref{identity on Phihat and Psihat} along all smooth curves in $D$ intersecting transversally with $S$, we obtain
\begin{equation*}
	\widehat{\Phi}=O(\rho^{n+2}), \qquad \widehat{\Psi}=O(\rho^{N-1}).
\end{equation*}
Combining them with \eqref{Phihat and Psihat def}, and replacing $N$ by $N+2$, we achieve \eqref{relation between Phi and A} and \eqref{relation between Psi and B}.

To prove the uniqueness of $\{\alpha_j\}_{j=0}^{n+1}$, we suppose $\{\widehat{\alpha}_j\}_{j=0}^{n+1}$ is another family of smooth functions on $M$ such that
\begin{equation*}
	\Phi=\sum_{j=0}^{n+1}(\widehat{\alpha}_j \circ \pi)\rho^j +O(\rho^{n+2}).
\end{equation*}
Then this together with \eqref{relation between Phi and A} imply
\begin{equation*}
	\sum_{j=0}^{n+1}\bigl(\alpha_j \circ \pi-\widehat{\alpha}_j \circ \pi\bigr)\rho^j=O(\rho^{n+2}).
\end{equation*}
Comparing the vanishing order of both sides on $S$, we get
\begin{equation*}
	\alpha_j \circ \pi =\widehat{\alpha}_j \circ \pi \text{ on } S, \quad \forall~~ 0\leq j\leq n+1.
\end{equation*}
But $\pi: S \rightarrow M$ is surjective. We conclude $\alpha_j$ equals $\widehat{\alpha}_j$ as functions on $M$. The uniqueness of $\beta_j$ also follows by a similar argument. This finishes the proof of the first assertion in Theorem \ref{Thm the coeffiecients relation}.

It is clear from the above proof that $\alpha_j$'s and $\beta_j$'s  can be explicitly computed in terms of $a_j$'s by \eqref{relation between A and a} and \eqref{relation between B and a}.  To prove the last conclusion in Theorem \ref{Thm the coeffiecients relation}, we note \eqref{relation between B and a} implies
\begin{equation*}
	(\beta_0, \cdots, \beta_j)^{\intercal}=M(a_{n+2}, \cdots, a_{n+2+j})^{\intercal}
\end{equation*}
for some lower-triangular matrix $M$ with diagonal entries $\frac{1}{(2\pi)^{n+1}}\tau_{ii}=\frac{(-1)^{i+1}}{(2\pi)^{n+1}\, i!}\neq 0$. Consequently, conditions (i) and (ii) in Theorem \ref{Thm the coeffiecients relation} are equivalent. This finishes the proof of Theorem \ref{Thm the coeffiecients relation}.
\end{proof}

\section{Proof of Theorem \ref{Thm BLF criterion}, Corollary \ref{Cor BLF criterion compact} and Corollary \ref{Cor cpt homogeneous is BLF}}\label{Sec 4}

We first note that Theorem \ref{Thm BLF criterion} is an immediate consequence of Corollary \ref{Cor Fefferman expansion} and Theorem \ref{Thm the coeffiecients relation}.
\begin{proof}[Proof of Theorem $\ref{Thm BLF criterion}$]
Note that condition (ii), by Corollary \ref{Cor Fefferman expansion},  is equivalent to the following:
	\begin{itemize}
		\item[(iii)] $\Psi$ vanishes to infinite order along $\Sigma$.
	\end{itemize}
Here (iii) means, in terms of the notation in Theorem \ref{Thm the coeffiecients relation} (see equation \eqref{relation between Psi and B} there), $\beta_j=0$ on $U$ for all $j\geq 0$. By the last assertion of Theorem \ref{Thm the coeffiecients relation}, this is equivalent to condition (i) in Theorem \ref{Thm BLF criterion}. The proof of Theorem \ref{Thm BLF criterion} is complete.
\end{proof}

Next, we prove Corollary \ref{Cor BLF criterion compact}.
\begin{proof}[Proof of Corollary $\ref{Cor BLF criterion compact}$]
We first prove the equivalence of \text{(i)} and \text{(ii)}. The implication ``$\text{(ii)}\implies \text{(i)}$'' follows immediately from Theorem \ref{Thm BLF criterion}. It remains to prove the converse. For that, we assume (i) holds, i.e., $a_{n+2+m}$ is constant on $M$ for every $m\geq 0$. By Proposition \ref{Prop TYZ expansion}, we have
	\begin{equation*}
		B_k(z)\sim \bigg(\frac{k}{\pi}\bigg)^n\sum_{j=0}^{\infty} \frac{a_j(z)}{k^j}, ~\text{ locally uniformly for } z\in M.
	\end{equation*}
	Now that $M$ is compact, this implies that (see Remark \ref{rmk asymptotic expansion meaning}): For any $N\geq 0$, there exists a constant $C_N>0$, independent of $k$, such that
	\begin{equation*}
		\bigg|B_k(z)-\bigg(\frac{k}{\pi}\bigg)^n\sum_{j=0}^{N} \frac{a_j}{k^j} \bigg|\leq C_N k^{n-N-1}, \quad \text{for any $k\geq 1$ and $z\in M$}.
	\end{equation*}
	Taking the integral over $M$, we get for any $k\geq 1$,
	\begin{multline}\label{TYZ expansion integral}
		\Biggl|\int_M B_k(z)dV_g-\frac{1}{\pi^n}\sum_{j=0}^{N} \Bigl(\int_M a_jdV_g\Bigr)k^{n-j} \Biggr|\leq \\ \int_M \Bigl| B_k(z)-\bigg(\frac{k}{\pi}\bigg)^n\sum_{j=0}^{N}  a_j k^{-j} \Bigr| dV_g\leq C_N\Vol_M k^{n-N-1}.
	\end{multline}
Here, $dV_g$ is the volume form of the metric $g$ on $M$, and $\Vol_M$ denotes the volume of $M$ with respect to $dV_g.$ 
On the other hand, by the definition of $B_{k}$ we have
\begin{equation}\label{eqn bk h0}
	\dim H^0(M,  L^{k} \otimes \mathcal{C}_M)= \dim A^2(M,  L^{k} \otimes \mathcal{C}_M) = \int_M B_{k}(z) dV_g,
\end{equation}	
where, as above,  $\mathcal{C}_M$ denotes the canonical line bundle of $M$. Since $L^{k}$ is positive, by the Kodaira vanishing theorem, we get
\begin{equation*}\label{Kodaira vanishing}
	H^q(M,  L^{k} \otimes \mathcal{C}_M)=0, \quad \forall q\geq 1.
\end{equation*}
Recall that the holomorphic Euler characteristic of $E:=L^{k} \otimes \mathcal{C}_M$ in sheaf cohomology is given by
\begin{equation*}\label{Euler characteristic}
	\chi(M, E)=\sum_{j=0}^n (-1)^j \dim_{\mathbb{C}} H^i(M, E).
\end{equation*}	
Thus, we have 
\begin{equation}\label{eqn chi h0}
	\chi(M, E)=\dim_{\mathbb{C}} H^0(M, E)=\int_M B_{k}(z) dV_g.
\end{equation}
Moreover, by the Riemann--Roch--Hirzebruch theorem, it follows that
\begin{equation}\label{Riemann-Roch-Hirzebruch}
	\chi(M, E)=\int_M \mathrm{Ch}(E)\wedge \mathrm{Td}(M)=\sum_{j=0}^n \int_M \mathrm{Ch}_{n-j}(E)\wedge \mathrm{Td}_j(M),
\end{equation}
where $\mathrm{Ch}(E)$ is the Chern character of $E$, $\mathrm{Td}(M)$ the Todd class of the tangent bundle of $M$, $\mathrm{Ch}_{n-j}(E)$  the $(n-j, n-j)$-form part of $\mathrm{Ch}(E)$, and $\mathrm{Td}_j(M)$ the $(j, j)$-form part of $\mathrm{Td}(M)$.
We also recall that
\begin{equation*}
	\mathrm{Ch}(E)=e^{c_1(E)}=\sum_{l=0}^{\infty} \frac{c_1(E)^l}{l!}=\sum_{l=0}^{\infty} \frac{\bigl(-c_1(M)+k\,c_1(L) \bigr)^l}{l!}.
\end{equation*}
Combining this with \eqref{Riemann-Roch-Hirzebruch}, we obtain
\begin{align*}
	\chi(M, E)=\sum_{j=0}^n\int_M \frac{\bigl(-c_1(M)+k\,c_1(L)\bigr)^{n-j}}{(n-j)!} \mathrm{Td}_j(M),
\end{align*}
which is a polynomial in $k$ of degree $n$. Consequently, by \eqref{eqn chi h0}, $\int_M B_k(z) dV_g$
is also a polynomial in $k$ of degree $n$. Let us express this as
	\begin{equation*}
		\int_M B_k(z) dV_g= \lambda_n k^n+\cdots+\lambda_1k+\lambda_0,
	\end{equation*}
	for some constants $\lambda_j$. Putting this into \eqref{TYZ expansion integral} and taking $N>n+1$, we conclude that
	\begin{equation*}
		\frac{1}{\pi^n}\int_M a_j dV_g=\lambda_{n-j}, \quad \forall 0\leq j\leq n \quad \text{and} \quad \int_M a_{n+1}dV_g=\cdots=\int_M a_{N}dV_g=0.
	\end{equation*}
	Since $N>n+1$ is arbitrary, we obtain $\int_Ma_{n+1+m}dV_g=0$ for all $m\geq 0$. But by the assumption in (i), every $a_{n+2+m}$ is constant on $M$. This implies $a_{n+2+m} \equiv 0$ on $M$ for all $m\geq 0$. By Theorem \ref{Thm BLF criterion}, $S$ is Bergman logarithmically flat, i.e., condition (ii). This proves the equivalence of \text{(i)} and \text{(ii)}.
	
We finally consider condition \text{(iii)}. As proved above, we always have $\int_M a_{n+2+m}\, dV_g = 0$ for all $m \geq 0.$ Then the sign condition in \text{(iii)} is equivalent to $a_{n+2+m} \equiv 0$ on $M$ for all $m \geq 0$. By Theorem \ref{Thm BLF criterion}, this is further equivalent to $S$ being Bergman logarithmically flat. Hence \text{(iii)} is equivalent to \text{(ii)}. This completes the proof of Corollary \ref{Cor BLF criterion compact}.
\end{proof}

\begin{rmk}
Note that in the proof of Corollary \ref{Cor BLF criterion compact}, we obtained the additional conclusion
$\int_M a_{n+1}\, dV_g = 0,$
which was not needed in establishing the corollary. We also note that, in fact, we established the vanishing of the integrals $\int_M a_{n+1+m}\, dV_g$, for all $m\geq 0$, without any of the (equivalent) assumptions (i)-(iii). We state this result as a separate proposition.
\end{rmk}

\begin{prop}\label{Prop compactM}
Let $(M, g; L,h)$ be a polarized compact manifold. Let $\{a_j\}_{j=0}^{\infty}$ be
the coefficient functions appearing in the expansion \eqref{TYZ expansion intro} of the Bergman kernel function $B_k$. Then,
\begin{equation}
\int_M a_{n+1+m}\, dV_g=0,\qquad \forall m\geq 0.
\end{equation}
\end{prop}

We conclude this section by proving Corollary \ref{Cor cpt homogeneous is BLF}.
\begin{proof}[Proof of Corollary $\ref{Cor cpt homogeneous is BLF}$]
Recall that, by Proposition \ref{Prop TYZ expansion}, the functions $a_j$ in \eqref{TYZ expansion intro} are polynomials in the curvature of $(M, g)$ and its covariant derivatives. As a consequence, if $(M, g)$ is locally homogeneous, then each $a_j$ for $j\geq 0$ must be constant on $M$. Hence, by Corollary \ref{Cor BLF criterion compact},
$S$ is Bergman logarithmically flat. Finally, the last assertion in Corollary \ref{Cor cpt homogeneous is BLF} follows immediately from the work of Bryant \cite{Bry01}, Webster \cite{Web78}, and Wang \cite[Theorem 12]{Wang} (see also, e.g., \cite[Proposition 1.12]{EXX23}). This finishes the proof of Corollary \ref{Cor cpt homogeneous is BLF}.
\end{proof}
	

\section{Appendix. A more general form of Proposition \ref{Prop TYZ expansion} and a proof}\label{Sec Appendix}
In this appendix, we first prove the following proposition and then show that Proposition \ref{Prop TYZ expansion} follows from it as a special case. We establish this more general version because it will be utilized in a subsequent paper on the Szeg\"o kernel \cite{EXX25}.
\begin{prop}\label{Prop TYZ expansion Appendix}
	Let $(M, g; L,h)$ be a polarized manifold, and assume that $M$ admits some complete \k metric. Let $(L', h')$ be another (not necessarily positive) Hermitian line bundle over $M$. Suppose there exists some constant $C_0 \geq 0$ such that
	\begin{equation}\label{curvature assumption lmg}
		\Ric(L', h')+\Ric(M, g) \geq -C_0 g \quad \text{on } M.
	\end{equation}
	 Consider the Bergman kernel function $B_k=B_{E_k}$ of $(E_k=L^k\otimes L', h_k=h^k\otimes h'),$ where $k\in \mathbb{Z}^+$. Then, there exist unique smooth functions $a_j$ for $j\geq 0$ such that $B_k$ has the following asymptotic expansion:
	\begin{equation}\label{TYZ expansion Appendix}
		B_k(z)\sim \bigg(\frac{k}{\pi}\bigg)^n\sum_{j=0}^{\infty} \frac{a_j(z)}{k^j}, ~\text{ locally uniformly for } z\in M.
	\end{equation}
	Moreover, the functions $a_{j}$ are all universal polynomials in the curvatures of $(M, g)$ and $(L', h')$ together with their covariant derivatives with respect to the Levi-Civita and Chern connections, respectively.
\end{prop}

\begin{rmk}\label{rmk asymptotic expansion meaning appendix}
As in Remark \ref{rmk asymptotic expansion meaning} above, the asymptotic expansion \eqref{TYZ expansion Appendix} means that, for any $N\geq 0$, $m\geq 0$ and $X \subset\subset M$, there exists $C_{N, m, X}$ depending on $N, m, X$, $(M, g), (L, h)$ and $( L', h'),$ while independent of $k$, such that
	\begin{equation*}
		\Bigg\|B_k-\bigg(\frac{k}{\pi}\bigg)^n\sum_{j=0}^{N} \frac{a_j(z)}{k^j} \Bigg\|_{C^m(X)}\leq C_{N, m, X} k^{n-N-1}, \quad \forall k\geq 1.
	\end{equation*}
\end{rmk}

To establish Proposition \ref{Prop TYZ expansion Appendix}, we will follow the framework of \cite{BBSj08} (see also \cite[Appendix]{DHSj22} and \cite[Section 5]{HS22}).
As a consequence of the $L^2$ estimates in \cite[Theorem 6.1]{Dem97}, their method can be extended to general complex manifolds that possess a complete \k metric. Here, we make only the necessary minor adjustments to handle this general setting.


Let us first recall the local reproducing kernel.
Fix $p\in M$. Let $U, V$ be small neighborhoods of $p$ such that $p\in V\subset\subset U \subset\subset M$. Shrinking them if needed, we can assume $(U, z)$ is a coordinate chart of $(M,g; L,h)$ and $(L', h')$ with a frame $e_{L}$ of $L$ and a frame $e_{L'}$ of $L'$ over $U$. We denote by $d\mu(z)$ the Lebesgue measure on $U$, with $dV_g(z)=\det g(z)\, d\mu(z)$.
We write $e^{-\phi(z)}=h(e_L(z), e_{L}(z))$, $e^{-\phi'(z)}=h'(e_{L'}(z), e_{L'}(z))$ and  define the local $L^2$-space and the local Bergman space as follows:
\begin{align*}
	L^2(U, e^{-k\phi-\phi'})=&\left\{f: U\rightarrow \mathbb{C} \text{ is measurable} \colon \int_U |f|^2 e^{-k\phi-\phi'} dV_g<\infty \right\},\\
	A^2(U, e^{-k\phi-\phi'})=&\left\{f\in L^2(U, e^{-k\phi-\phi'}) \colon f \text{ is holomorphic on } U\right\}.
\end{align*}
We will also write $A^2(U, e^{-k\phi-\phi'})$ as $A^2(U)$ for brevity.

Throughout this section, let $\chi$ be a smooth cut-off function satisfying $0\leq \chi\leq 1$, $\chi|_V  \equiv 1$ and $\supp\chi\subset\subset U$. The following notion of a local reproducing kernel was introduced in \cite{BBSj08}.
\begin{defn}\label{defn local Bergman kernel}{\rm
	A family of functions $\widetilde{K}^{(N)}_k(x, \oo{y})$ on $U\times U$ is said to be a {\em local reproducing kernel} $\mod O(\frac{1}{k^{N+1-n}})$ if for any $u\in A^2(U, e^{-k\phi-\phi'})$ and $x\in V$ we have
	\begin{equation}\label{approximate reproducing property}
		u(x)=\int_U \chi(y) u(y) \widetilde{K}^{(N)}_k(x, \oo{y}) e^{-k\phi(y)-\phi'(y)} dV_g(y)+O\bigg(\frac{1}{k^{N+1-n}}\bigg) e^{\frac{k}{2}\phi(x)} \|u\|_{A^2(U)},
	\end{equation}
uniformly on compact subsets of $V$. }
\end{defn}
Furthermore, \cite{BBSj08} provided an explicit construction of a local reproducing kernel.
\begin{prop}[{\cite[Proposition 2.7 and Section 2.5]{BBSj08}}]\label{Prop local reproducing kernel}
	There exists a local reproducing kernel $\widetilde{K}^{(N)}_k(x, \oo{y}) \mod O(\frac{1}{k^{N+1-n}})$ such that
	\begin{equation}\label{local reproducing kernel}
		\widetilde{K}^{(N)}_k(x, \oo{y})=\bigg(\frac{k}{\pi}\bigg)^n e^{k\psi(x, \oo{y})+\psi'(x, \oo{y})} \left(a_0(x, \oo{y})+\frac{a_1(x, \oo{y})}{k}+\cdots +\frac{a_N(x, \oo{y})}{k^N} \right),
	\end{equation}
	where $\psi(x, \oo{y})$ and $\psi'(x, \oo{y})$ are respectively the almost holomorphic extensions of $\phi(x)$ and $\phi'(x)$, and each $a_j(x, \oo{y})$ is a  universal polynomial in the derivatives of $\psi(x, \oo{y})$ and $\psi'(x, \oo{y})$.
\end{prop}

\begin{rmk}\label{RMK TYZ coefficients are curvatures}
	Note that the almost holomorphic extensions $\psi(x, \oo{y})$ and  $\psi'(x, \oo{y})$ of $\phi(x)$ and $\phi'(x)$ are only unique up to $O(|x-y|^{\infty})$, and thus $a_j(x, \oo{y})$ may depend on the choice of the almost holomorphic extension $\psi$ and $\psi'$.  On the other hand, when restricted to the diagonal $y=x$, every $a_j(x):=a_j(x, \oo{x})$ is unique.
	As explained in \cite[Section 2.6]{BBSj08}, for different choices of almost holomorphic extensions of $\phi$ and $\phi'$, the resulted different choices of $\widetilde{K}^{(N)}_k(x, \oo{y})$ do not affect the local reproducing property \eqref{approximate reproducing property}. Without loss of generality, we will assume that the almost holomorphic extension $\psi$ satisfies $\oo{\psi(x, \oo{y})}=\psi(y, \oo{x})$, since we can replace $\psi(x, \oo{y})$ by $\frac{1}{2}(\psi(x, \oo{y})+\oo{\psi(y, \oo{x})})$ if needed.
\end{rmk}

Let $(U, z)$ and $e_L$ be as introduced before Definition \ref{defn local Bergman kernel}.  Let $k \geq 1$ and $K(x, y)=K_k(x, y)$ be the Bergman kernel of $L^k\otimes L'$; and for $y \in U$, write $K(\cdot, y)=\widehat{K}(\cdot, \oo{y})  \oo{e_L^k(y)}\otimes \oo{e_{L'}(y)}$ for $x, y\in U$, where $\widehat{K}(\cdot, \oo{y}) \in H^0(M, L^k\otimes L').$ We further write $\widehat{K}(x, \oo{y})=\widetilde{K}(x, \oo{y}) e_L^k(x) \otimes e_{L'}(x)$ for $x \in U$.  Note that we denote the global Bergman kernel of the line bundle by $K(x,y)$, with no bar over the second variable $y$. On the other hand, in local frames, we denote the corresponding sections and coefficient functions by $\widehat{K}(x, \oo{y})$ and $\widetilde{K}(x,\overline{y})$ respectively, with a bar on the second variable $y$. This convention is consistent with the notation in \cite{BBSj08}. We also note that, for convenience of notation, we have dropped the subscript $k$, which indicates the dependence on $k$. In order to be consistent, we shall also drop the subscript $k$ on the local reproducing kernel and use the notation $\widetilde{K}^{(N)}(x, \oo{y})=\widetilde{K}^{(N)}_k(x, \oo{y})$.

In the following theorem, we show that $\widetilde{K}$ is equal to the local reproducing kernel up to a small error term on a small precompact subset of $V$.
\begin{thm}\label{Thm derivative estimates on the BK}
Let $(M, g; L,h)$ be a polarized manifold, and assume that $M$ admits some complete metric. Let $(L', h')$ be another (not necessarily positive) Hermitian line bundle over $M$.
Suppose there exists some constant $C_0 \geq 0$ such that
	\begin{equation}\label{curvature assumption in TYZ expansion}
		\Ric(L', h')+\Ric(M, g) \geq -C_0 g \quad \text{on } M.
	\end{equation} Let $p\in M$ and $W\subset\subset V\subset\subset U\subset\subset M$ be small neighborhoods of $p$. Then, for any $N \geq 1$ and $\alpha, \beta \in \mathbb{Z}_{\geq 0}^n$, there exists a constant $C=C(U, V, W, N, \alpha, \beta)>0$ such that for any $k\geq 1$,
	\begin{equation*}
		\left|D_x^{\alpha} D_{\oo{y}}^{\beta} \left(\widetilde{K}(x, \oo{y})-\widetilde{K}^{(N)}(x, \oo{y}) \right)\right| e^{-\frac{k}{2}\phi(x)-\frac{k}{2}\phi(y)} \leq \frac{C}{k^{N+1-n-|\alpha|-|\beta|}}, \quad \forall x, y\in W.
	\end{equation*}
\end{thm}

In the following proof, we shall follow the standard convention of allowing constants in an estimate to change from line to line without renaming them---The same symbol $C$ (and similarly for $C_m, C_N$, etc.) may denote different constants in different places, though always independent of $k$.

\begin{proof}[Proof of Theorem $\ref{Thm derivative estimates on the BK}$]
	Fix $y\in U$. By the definition of the Bergman kernel $K$, we observe that  $\widehat{K}(\cdot , \oo{y})$ belongs to $A^2(M, L^k\otimes L')$ and
	\begin{equation} \label{eqn hatk}
		\int_{M} |\widehat{K} (\cdot, \oo{y})|^2_{h^k \otimes h'} dV_g = \left(\widehat{K} (\cdot, \oo{y}), \widehat{K}(\cdot, \oo{y})\right)_{L^2(M, L^k\otimes L')}= \widetilde{K}(y, \oo{y}) < \infty.
	\end{equation}
	Moreover, we note that $|\widehat{K} (x, \oo{y})|^2_{h^k \otimes h'}=|\widetilde{K}(x, \oo{y})|^2 e^{-k\phi(x)-\phi'(x)}$ for $x \in U$. This together with \eqref{eqn hatk} implies
	\begin{equation}\label{eqn widetildek a2}
		\widetilde{K}(\cdot, \oo{y})\in A^2(U) \quad \text{and} \quad \|\widetilde{K}(\cdot, \oo{y})\|^2_{A^2(U)}\leq \widetilde{K}(y, \oo{y}).
	\end{equation}
	Then, applying the asymptotic reproducing property \eqref{approximate reproducing property} of $\widetilde{K}^{(N)}$ with $u=\widetilde{K}(\cdot, \oo{y})$, we obtain
	\begin{equation}\label{eqn u equal widetildek}
		\widetilde{K}(x, \oo{y})=\int_U \chi(z) \widetilde{K}(z, \oo{y}) \widetilde{K}^{(N)}(x, \oo{z}) e^{-k\phi(z)-\phi'(z)} dV_g(z)+O\Bigl(\frac{1}{k^{N+1-n}}\Bigr) e^{\frac{k}{2}\phi(x)} \|\widetilde{K}(\cdot, \oo{y})\|_{A^2(U)}.
	\end{equation}
	Recall that the Bergman kernel function is given by $B_k(y)=|K(y,y)|_{h^k \otimes h'}=\widetilde{K}(y, \oo{y})e^{-k\phi(y)-\phi'(y)}.$  It follows from the extremal characterization of the Bergman function and the sub-mean value inequality for holomorphic sections over ball of radius roughly $\frac{1}{\sqrt{k}}$ that
	\begin{equation}\label{Bergman function upper bound}
		B_k(y)\leq Ck^n
	\end{equation}
	uniformly on $U\subset\subset M$ (see \cite[Equation (2.5)]{Ber03} and \cite[Lemma 4.1]{HKSX16}). Consequently, by \eqref{eqn widetildek a2}, $\|\widetilde{K}(\cdot, \oo{y})\|^2_{A^2(U)}$ is bounded from above by $Ck^ne^{k\phi(y)+\phi'(y)}.$ This together with \eqref{eqn u equal widetildek} yields
	\begin{equation}\label{local to global term 1}
		\widetilde{K}(x, \oo{y})=\int_U \chi(z) \widetilde{K}(z, \oo{y}) \widetilde{K}^{(N)}(x, \oo{z}) e^{-k\phi(z)-\phi'(z)} dV_g(z)+O\Bigg(\frac{1}{k^{N+1-\frac{3}{2}n}}\Bigg) e^{\frac{k}{2}\phi(x)+\frac{k}{2}\phi(y)}.
	\end{equation}
	
	Let $\Pi: L^2(M, L^k\otimes L') \rightarrow A^2(M, L^k\otimes L')$ be the Bergman projection. Next, for each fixed $x \in W$, we first define a smooth section $Q_x$ of $L^k\otimes L'$ over $M$, compactly supported in $U$, by
\begin{equation}\label{eqn qx}
Q_x(y)= \chi(y) \oo{\widetilde{K}^{(N)}(x, \oo{y})}~e_L^k(y)\otimes e_{L'}(y),  \quad \forall y \in U.
\end{equation}
We then define the smooth section $u_x$ of $L^k\otimes L'$ over $M$ as
	\begin{equation}\label{eqn defn ux}
		u_x(w)=Q_x(w)-\Pi(Q_x)(w), \quad \forall w \in M.
	\end{equation}
When $w=y\in U$, we can write $\Pi$ as an integral in local coordinates and obtain $u_x(y)=\widetilde{u}_x(y) e_L^k(y)\otimes e_{L'}(y)$, where
	\begin{equation}\label{u_x definition}
		\widetilde{u}_x(y)=\chi(y) \oo{\widetilde{K}^{(N)}(x, \oo{y})}-\int_U \chi(z) \oo{\widetilde{K}^{(N)}(x, \oo{z})}\widetilde{K}(y, \oo{z})  e^{-k\phi(z)-\phi'(z)} dV_g(z).
	\end{equation}
Note by \eqref{eqn defn ux}, $\dbar u_x=\dbar Q_x$. Moreover, $u=u_x$ is the $L^2$-minimal solution of the $\dbar$-equation
	\begin{equation}\label{dbar equation}
		\dbar u=\dbar Q_x.
	\end{equation}
	It is easy to see $\dbar u_x$ is compactly supported in $U$, and thus $\dbar u_x\in L^2_{(0, 1)}(M, L^k\otimes L')$. Moreover, for $w=y\in U$, by \eqref{eqn qx} $\dbar u_x$ further equals
	\begin{equation}\label{two terms in the dbar equation}
		\dbar \widetilde{u}_x(y) \wedge e_L^k(y)\otimes e_{L'}(y) = \left[ \big(\dbar \chi(y)\big) \, \oo{\widetilde{K}^{(N)}(x, \oo{y})}+\chi(y)\, \dbar\oo{\widetilde{K}^{(N)}(x, \oo{y})}~\right]\wedge e_L^k(y)\otimes e_{L'}(y).
	\end{equation}
	By \eqref{local reproducing kernel}, shrinking $U$ if needed, there exists some positive constant $C=C(N)$ such that
	\begin{equation*}
		\left|\widetilde{K}^{(N)}(x, \oo{y})\right|e^{-\frac{k}{2}\phi(x)-\frac{k}{2}\phi(y)}\leq C k^ne^{-\frac{k}{2}D(x, y)}, \quad \forall x \in W, ~y\in  U,
	\end{equation*}
	where $D(x, y)=\phi(y)+\phi(x)-\psi(y, \oo{x})-\psi(x, \oo{y})$ is the Calabi's diastasis function. Since $\phi$ is strictly plurisubharnomic, we have $D(x, y) \geq 2\tau |x-y|^2$ on $U$ for some $\tau >0$.  Since $\chi=1$ on $V$, if $\dbar\chi(y)\neq 0$, then $y\in U-V$ and thus $|x-y|^2$ has a positive lower bound uniformly for $x\in W$. Therefore, there exist positive constants $C$ and $\delta$ such that $D(x, y) \geq 2 \delta$ and
	\begin{equation}\label{term 1 in dbar equation}
		\left|\big(\dbar \chi(y)\big) \, \oo{\widetilde{K}^{(N)}(x, \oo{y})}\right| e^{-\frac{k}{2}\phi(x)-\frac{k}{2}\phi(y)} \leq Ck^n e^{-\delta k}, \quad \forall x\in W, ~y\in U.
	\end{equation}
	
	On the other hand, by \eqref{local reproducing kernel} and the almost holomorphicity of $\psi(x, \oo{y})$, $\psi'(x, \oo{y})$ and each $a_j(x, \oo{y})$, for every $m>0$, there exists $C_m>0$ such that for any $x, y\in U$ (shrinking $U$ if needed), we have
	\begin{align*}
		\left|\dbar\oo{\widetilde{K}^{(N)}(x, \oo{y})}\right| e^{-\frac{k}{2}\phi(x)-\frac{k}{2}\phi(y)}\leq C_m k^ne^{-\frac{k}{2}D(x, y)} |x-y|^{2m}\leq C_m k^n e^{-\tau k|x-y|^2} |x-y|^{2m}.
	\end{align*}
	Given any $a>0$, we note that $\sup_{t\geq 0} e^{-at} t^m$, achieved at $t=\frac{m}{a}$, is $e^{-m}m^m a^{-m}$. Applying this fact to the above $e^{-\tau k|x-y|^2} |x-y|^{2m}$, we have
	\begin{align}\label{term 2 in dbar equation}
		\left|\dbar\oo{\widetilde{K}^{(N)}(x, \oo{y})}\right| e^{-\frac{k}{2}\phi(x)-\frac{k}{2}\phi(y)}\leq C_m k^{n-m}, \quad \forall x, y\in U.
	\end{align}
	By combining \eqref{two terms in the dbar equation},  \eqref{term 1 in dbar equation} and \eqref{term 2 in dbar equation}, for any $m>0$, we get
	\begin{equation}\label{u_x dbar pointwise estimate}
		|\dbar \widetilde{u}_x(y)|e^{-\frac{k}{2}\phi(x)-\frac{k}{2}\phi(y)}\leq Ck^n \left(e^{-\delta k}+C_mk^{-m}\right) \leq C_m k^{n-m}, \quad \forall y\in U, x\in W.
	\end{equation}
	Here to obtain the last inequality, we have used the fact that $e^{-\delta k}=O(\frac{1}{k^{\infty}})$ as $k\rightarrow \infty$.
	The following version of $L^2$ estimates follows directly from \cite[Chapter VIII, Theorem 6.1]{Dem97}.
	\begin{prop}\label{prop dem}
		Let $X$ be a complex manifold of dimension $n$, which admits some complete \k metric. Equip $X$ with a \k metric $g_0$, which is not necessarily complete. Let $(L_X, h_X)$ be a semi-positive line bundle over $X$. Then for any $v\in L^2_{(n, 1)}(X, L_X)$ with $\dbar v=0$ and $\int_X \langle A^{-1}v, v \rangle dV_{g_0}<\infty$, there exists $u\in L^2_{(n,0)}(X, L_X)$ such that $\dbar u=v$ and
		\begin{equation}\label{eqn dbar estimate}
			\int_X \langle u, u \rangle dV_{g_0} \leq \int_X \langle A^{-1}v, v \rangle dV_{g_0}.
		\end{equation}
		Here all pointwise Hermitian inner products $\langle \cdot, \cdot \rangle$ are taken with respect to  $g_0$ and $h_X$; and $A$ is a linear operator on $L^2_{(n, 1)}(X, L_X)$ defined in terms of $\Ric(L_X, h_X)$ and $g_0$ as in \cite[Chapter VIII, Theorem 6.1]{Dem97}.
	\end{prop}
Let $x \in W$ be fixed and $u_x$ be defined as in \eqref{eqn defn ux}.
Recall by the discussion above, $u_x$ is the $L^2$-minimal solution of the $\dbar$-equation $\dbar u=v:=\dbar u_x \in L^2_{(0, 1)}(M, L^k\otimes L').$ Moreover, note that
$$L^2(M, L^k\otimes L')=L^2_{(n, 0)}(M, L^k\otimes L'\otimes \mathcal{C}_M^{-1}), \quad L^2_{(0, 1)}(M, L^k\otimes L')=L^2_{(n, 1)}(M, L^k\otimes L'\otimes \mathcal{C}_M^{-1}).$$
To apply Proposition \ref{prop dem} to the $\dbar$-equation $\dbar u=v$, we set $X=M$, $g_0=g$ and $(L_X, h_X)=(L^k\otimes L'\otimes \mathcal{C}_M^{-1}, h^k\otimes h'\otimes (\det g))$. We make a key observation as follows: under the assumption \eqref{curvature assumption in TYZ expansion}, we have, when  $k>C_0$,  that $A$ is bounded below by $k-C_0$, and thus $A^{-1}$ is bounded above by $(k-C_0)^{-1}$. As before, write $u_x(y)=\widetilde{u}_x(y) e_L^k(y)\otimes e_{L'}(y)$. Then, since $u_x$ is the $L^2$-minimal solution of $\dbar u=v$,  we obtain by \eqref{eqn dbar estimate} for $k >C_0$,
	\begin{align*}
		\int_U |\widetilde{u}_x(y)|^2 e^{-k\phi(y)-\phi'(y)} dV_g(y)\leq&  \frac{1}{k-C_0} \|\dbar u_x\|^2_{L^2_{(0, 1)}(M, L^k\otimes L')}\\		
		=& \frac{1}{k-C_0} \int_U |\dbar \widetilde{u}_x(y)|_g^2 e^{-k\phi(y)-\phi'(y)} dV_g(y).
	\end{align*}
Shrinking $U$ if needed, we have that $\phi'$ is bounded on $U\subset\subset M$, and $g$ is comparable to the Euclidean metric on $U\subset\subset M$. Using these facts and \eqref{u_x dbar pointwise estimate}, we further have
	\begin{equation}\label{u_x L2 estimates}
		\left(\int_U |\widetilde{u}_x(y)|^2 e^{-k\phi(y)} d\mu(y)\right)^{\frac{1}{2}}\leq C_mk^{n-m} e^{\frac{k}{2}\phi(x)}, \quad \forall k \geq 1.
	\end{equation}
Here $C_m$ is taken sufficiently large to cover also the cases $1 \leq k \leq C_0$. Let $r:=d(W, M-U)$. By \cite[Proposition 5.1]{HS22}, for $x, y\in W\subset\subset V\subset\subset U$ we get the pointwise estimate
	\begin{align*}
		|\widetilde{u}_x(y)|e^{-\frac{k}{2}\phi(y)}\leq C \left(\frac{r}{k}\sup_{z\in U} \Bigl(\bigl|\dbar \widetilde{u}_x(z)\bigr|e^{-\frac{k\phi(z)}{2}}\Bigr) +\frac{k^n}{r^n} \biggl(\int_U |\widetilde{u}_x(z)|^2 e^{-k\phi(z)} d\mu(z)\biggr)^{\frac{1}{2}}\right).
	\end{align*}
	Note that the constants $r$ and $r^{-n}$ can be absorbed into $C$. In view of \eqref{u_x dbar pointwise estimate} and \eqref{u_x L2 estimates}, we have
	\begin{align*}
		|\widetilde{u}_x(y)|e^{-\frac{k}{2}\phi(y)}\leq  C_m {k^{2n-m}} e^{\frac{k}{2}\phi(x)}, \quad \forall x, y\in W.
	\end{align*}
	Taking $m=N+1$, we have
	\begin{equation}\label{u_x_asymptotics}
		|\widetilde{u}_x(y)|\leq \frac{C_N}{k^{N+1-2n}} e^{\frac{k}{2}\phi(x)+\frac{k}{2}\phi(y)}, \quad \forall x, y\in W.
	\end{equation}
Note that $\chi\equiv 1$ on $V$. It follows from \eqref{local to global term 1} and \eqref{u_x definition} that, for $x, y\in W\subset\subset V$, we have
\begin{equation*}
		\widetilde{u}_x(y)= \widetilde{K}(x, \oo{y})-\widetilde{K}^{(N)}(x, \oo{y}) +O\Bigg(\frac{1}{k^{N+1-\frac{3}{2}n}}\Bigg) e^{\frac{k}{2}\phi(x)+\frac{k}{2}\phi(y)}. 
	\end{equation*}
Thus, we conclude from \eqref{u_x_asymptotics} that
	\begin{equation*}
		\left| \widetilde{K}(x, \oo{y})-\widetilde{K}^{(N)}(x, \oo{y}) \right| e^{-\frac{k}{2}\phi(x)-\frac{k}{2}\phi(y)} \leq \frac{C_N}{k^{N+1-2n}}, \quad \forall x, y\in W.
	\end{equation*}
	Note that this estimate is still off by $k^n$ from our desired result. To resolve this issue, we replace $N$ by $N+n$ and get
	\begin{equation*}
		\Bigg| \widetilde{K}(x, \oo{y})-\widetilde{K}^{(N)}(x, \oo{y})- \bigg(\frac{k}{\pi}\bigg)^n e^{k\psi(x, \oo{y})+\psi'(x, \oo{y})} \sum_{j=N+1}^{N+n} \frac{a_j(x, \oo{y})}{k^j}  \Bigg| e^{-\frac{k}{2}\phi(x)-\frac{k}{2}\phi(y)} \leq \frac{C_{N+n}}{k^{N+1-n}}.
	\end{equation*}
Then we note
	\begin{align*}
		\Bigg| \bigg(\frac{k}{\pi}\bigg)^n e^{k\psi(x, \oo{y})+\psi'(x, \oo{y})}\sum_{j=N+1}^{N+n} \frac{a_j(x, \oo{y})}{k^j}  \Bigg| e^{-\frac{k}{2}\phi(x)-\frac{k}{2}\phi(y)}=O\Bigg(\frac{1}{k^{N+1-n}}\Bigg), \quad \forall x, y\in W.
	\end{align*}
Here again we have used the fact that the Calabi diastasis function satisfies $D(x, y) \geq 2\tau |x-y|^2$ on $W$ for some $\tau >0$. Consequently, we have
	\begin{equation*}
		\left| \widetilde{K}(x, \oo{y})-\widetilde{K}^{(N)}(x, \oo{y}) \right| e^{-\frac{k}{2}\phi(x)-\frac{k}{2}\phi(y)} \leq \frac{C_N}{k^{N+1-n}}, \quad \forall x, y\in W.
	\end{equation*}
The derivative estimates now follow by inductively applying the Bochner--Martinelli formula on a ball of radius roughly $\frac{1}{k}$. Details of this induction are left to the reader. This finishes the proof of Theorem \ref{Thm derivative estimates on the BK}.
\end{proof}

We are now ready to prove Proposition $\ref{Prop TYZ expansion Appendix}$.
\begin{proof}[Proof of Proposition $\ref{Prop TYZ expansion Appendix}$]
Let $(U, z)$ be a trivializing coordinate chart  as introduced before Definition \ref{defn local Bergman kernel}. Let $\widetilde{K}$ and $\widetilde{K}^{(N)}$ be as in Theorem \ref{Thm derivative estimates on the BK}.
Recall that by \eqref{local reproducing kernel}
	\begin{equation*}
		B_{k}(x)=\widetilde{K}(x, \oo{x}) e^{-k\phi(x)-\phi'(x)}=\bigg(\frac{k}{\pi}\bigg)^n\sum_{j=0}^N \frac{a_j(x, \oo{x})}{k^j}+\bigl(\widetilde{K}(x, \oo{x})-K^{(N)}(x, \oo{x}) \bigr)e^{-k\phi(x)-\phi'(x)}.
	\end{equation*}
	Thus, it is sufficient to prove that for any multi-indices $\alpha$ and $\beta$, there exists some $C_{N, \alpha, \beta}>0$ such that for any $k\geq 1$,
	\begin{equation}\label{eqn Bergman kernel vs local kernel}
		\Biggl| D_x^{\alpha}D_{\oo{x}}^{\beta}\left[\Bigl(\widetilde{K}(x, \oo{x})-\widetilde{K}^{(N)}(x, \oo{x}) \Bigr)e^{-k\phi(x)-\phi'(x)}\right]\Biggr|\leq \frac{C_{N, \alpha, \beta}}{k^{N+1-n}} \quad \text{locally uniformly for } x\in M.
	\end{equation}
	By Theorem \ref{Thm derivative estimates on the BK}, for sufficiently small neighborhood $W\subset\subset M$, there exists $C=\widehat{C}_{N, \alpha, \beta}>0$ such that for any $k\geq 1$,
	\begin{equation*}
		\Biggl| D_x^{\alpha}D_{\oo{x}}^{\beta}\left[\Bigl(\widetilde{K}(x, \oo{x})-\widetilde{K}^{(N)}(x, \oo{x}) \Bigr)e^{-k\phi(x)-\phi'(x)}\right] \Biggr| \leq \frac{\widehat{C}_{N, \alpha, \beta} k^{|\alpha|+|\beta|}}{k^{N+1-n}}, \quad \forall x\in W.
	\end{equation*}
	This misses our goal by a factor of $k^{|\alpha|+|\beta|}$. To resolve this issue, we replace $N$ by $N+|\alpha|+|\beta|$ and obtain that
	\begin{equation*}
		\Biggl| D_x^{\alpha}D_{\oo{x}}^{\beta}\left[\Bigl(\widetilde{K}(x, \oo{x})-\widetilde{K}^{(N+|\alpha|+|\beta|)}(x, \oo{x}) \Bigr)e^{-k\phi(x)-\phi'(x)}\right] \Biggr| \leq \frac{\widehat{C}_{N+|\alpha|+|\beta|, \alpha, \beta}}{k^{N+1-n}}, \quad \forall x\in W.
	\end{equation*}
	Then the desired estimates follow by noting that
	\begin{align*}
		\Biggl| D_x^{\alpha}D_{\oo{x}}^{\beta}\left[\Bigl(\widetilde{K}^{(N+|\alpha|+|\beta|)}(x, \oo{x})-\widetilde{K}^{(N)}(x, \oo{x}) \Bigr)e^{-k\phi(x)-\phi'(x)}\right] \Biggr|
		&= \bigg(\frac{k}{\pi}\bigg)^n \left| D_x^{\alpha}D_{\oo{x}}^{\beta}\sum_{j=N+1}^{N+|\alpha|+|\beta|} \frac{a_{j}(x, \oo{x})}{k^j} \right|
		\\
		&\leq \frac{\widetilde{C}_{N, \alpha, \beta}}{k^{N+1-n}}, \quad \forall x\in W,
	\end{align*}
	for some positive constant $\widetilde{C}_{N, \alpha, \beta}$ independent of $k$. To finish the proof, we follow the arguments in \cite{Lu00} to show that each $a_j$ is a (universal) polynomial in the curvature of $(M, g)$, $(L', h')$ and their covariant derivatives. Indeed, by Proposition \ref{Prop local reproducing kernel} (see also Remark \ref{RMK TYZ coefficients are curvatures}), each $a_j(x)$ is a polynomial in the derivatives of $\phi(x)$ and $\phi'(x)$. Consequently, under $K$-coordinates and $K$-frames (special types of normal coordinates and normal frames, see \cite{Lu00}) centered at any point $x_0 \in M$, each $a_j(x_0)$ is a polynomial in derivatives of the form $\partial_x^{\alpha}\partial_{\oo{x}}^{\beta}\phi(x_0)$ and $\partial_x^{\alpha}\partial_{\oo{x}}^{\beta}\phi'(x_0)$ with $|\alpha| \geq 2$ and $|\beta| \geq 2$. Each such derivative $\partial_x^{\alpha}\partial_{\oo{x}}^{\beta}\phi(x_0)$ is a polynomial in the curvature of $(M, g)$ and its covariant derivatives at $x_0$, and similarly each derivative $\partial_x^{\alpha}\partial_{\oo{x}}^{\beta}\phi'(x_0)$ is a polynomial in the curvature of $(L', h')$ and its covariant derivatives at $x_0$. Hence, $a_j(x)$ is a polynomial in the curvature of $(M, g)$, $(L', h')$ and their covariant derivatives. Moreover, such a polynomial can be obtained by finitely many algebraic operations (see \cite[Theorem 1.1]{Lu00}).
\end{proof}

We shall conclude this section with the proof of Proposition \ref{Prop TYZ expansion}.
\begin{proof}[Proof of Proposition $\ref{Prop TYZ expansion}$]
	In order to apply Proposition \ref{Prop TYZ expansion Appendix}, we set $(L', h')=(\mathcal{C}_M, H)$. Note then
	\begin{equation}
		\Ric(L', h')+\Ric(M, g)=0 \quad \text{on } M,
	\end{equation}
	and thus the curvature assumption \eqref{curvature assumption lmg} is satisfied with  $C_0=0$. Therefore, the expansion \eqref{TYZ expansion intro} follows immediately. Moreover, by Proposition \ref{Prop TYZ expansion Appendix}, each function $a_j$ is polynomial in the curvature of $(M, g)$, $(\mathcal{C}_M, H)$ and their covariant derivatives. Since the curvature of $(\mathcal{C}_M, H)$ can also be expressed as a polynomial in the curvature of $(M, g)$, it follows that each $a_{j}$ is a polynomial in the curvature of $(M, g)$ and its covariant derivatives, as desired.
\end{proof}


\bibliographystyle{plain}
\bibliography{references}

\end{document}